\documentclass[10pt]{amsart}
\usepackage{array,youngtab}
\usepackage{longtable}
\usepackage{color, graphicx, comment}
\usepackage{bbm}
\usepackage{tikz}
\usetikzlibrary{automata,positioning}

\usepackage{mathtools}

\allowdisplaybreaks

\numberwithin{equation}{section} \overfullrule 5pt
\newtheorem{Theorem}{Theorem}[section]

\newtheorem{Conjecture}[Theorem]{Conjecture}
\newtheorem{Proposition}[Theorem]{Proposition}
\newtheorem{Lemma}[Theorem]{Lemma}
\newtheorem{Lemma and Definition}[Theorem]{Lemma and Definition}

\newtheorem{Observation}[Theorem]{Observation}
\theoremstyle{definition}

\newtheorem{Example}{Example}[section]
\newtheorem{Remark}{Remark}[section]

\newtheorem{manualtheoreminner}{Conjecture}
\newenvironment{manualtheorem}[1]{%
  \renewcommand\themanualtheoreminner{#1}%
  \manualtheoreminner
}{\endmanualtheoreminner}

\DeclareMathOperator{\Stiel}{Stiel}
\DeclareMathOperator{\CF}{CF}

\DeclareMathOperator{\Res}{Res}
\DeclareMathOperator{\F2}{\mathbb{F}_2}

\title[Algebraicity of Thue-Morse continued fractions
]{%
	On the algebraicity of Thue-Morse and period-doubling continued fractions 
}
\date{May. 23, 2020}

\author{Yining Hu}
\address{School of Mathematics and Statistics, 
Huazhong University of Science and Technology, Wuhan, PR China}
\email{huyining@hust.edu.cn}

\author{Guoniu Wei-Han}
\address{I.R.M.A., UMR 7501, Universit\'e de Strasbourg
et CNRS, 7 rue Ren\'e Descartes, 67084 Strasbourg, France}
\email{guoniu.han@unistra.fr}



\begin{document}
\begin{abstract} 
We put forward several general conjectures concerning the algebraicity or transcendence
of continued fractions and Stieltjes continued fractions defined by the 
	Thue-Morse and period-doubling sequences in characteristic $2$. 
We present our Guess'n'Prove method, in which we exploit the structure of 
	automata, for proving some of our conjectures in special cases. 
	\end{abstract}


\maketitle

\section{Introduction}
\subsection{Background}
We are interested in the continued fractions and Stieltjes continued fractions
defined by automatic sequences in finite characteristic, and more precisely their algebraicity or transcendence. We give here the background and motivation
for studying such problems. The definitions of related notions will given in subsection
\ref{prel}.

The link of automaticity and algebraicity goes back to the well-known Theorem
of Christol, Kamae, Mend\`es France and Rauzy \cite{Christol1980KMFR} 
which states that 
a formal power series in $\mathbb{F}_q[[x]]$ is algebraic over $\mathbb{F}_q(x)$
if and only if the sequence of its coefficients is $q$-automatic.
The situation is completely different for real numbers.
In 2007, Adamczewski and Bugeaud \cite{Adamczewski2007B} proved that for an integer 
$b\geq 2$, if 
the $b$-ary expansion of an irrational real number $u$ form an automatic 
sequence, then $u$ must be transcendental.
In 2013, Bugeaud \cite{Bugeaud2013} proved that the continued fraction 
expansion of an algebraic real number of degree at least $3$ is not automatic.

As with real numbers, a formal Laurent series can also be represented by
a continued fraction whose partial quotients are polynomials. 
Unlike for real numbers, the continued fraction expansion of an algebraic
Laurent series of degree at least $3$ may or may not have automatic partial
quotients \cite{Baum1976S, Baum1977S, Mills1986R, Allouche1988, Mkaouar1995, 
Lasjaunias2015Y,  Lasjaunias2016Y, Lasjaunias2017Y}; see also the introduction
of \cite{hh2}.
 
We could also ask the converse question: what can we say about 
the algebraicity of a continued fraction whose partial quotients form an 
automatic sequence? To our knowledge, little has been done in this direction. 
The authors \cite{hh2} proved that the Stieltjes continued
fractions defined by the Thue-Morse sequence and the period-doubling sequence
in $\mathbb{Z}[[x]]$ are congruent, modulo $4$, to algebraic series
in $\mathbb{Z}[[x]]$. 
In 2020, Wu \cite{Wu2020} obtained similar results concerning the Stieltjes
continued fractions defined by the paperfolding sequence and the Golay-Shaprio-Rudin sequence. 

In this article we propose to approach this problem with the most classical
example of automatic sequences, the Thue-Morse sequence.

\subsection{Preliminaries}\label{prel}
We introduce the necessary notions for stating the conjectures and our main results.
\subsubsection{Automatic sequnces}
A sequence is said to be {\it  $k$-automatic} if it can be generated by a $k$-DFAO ({\it deterministic finite automaton with output}).
For an integer $k\geq 2$,  a $k$-DFAO 
is defined to be a $6$-tuple
$$M=(Q,\Sigma, \delta, q_0, \Delta, \tau)$$
where $Q$ is the set of states with $q_0\in Q$ being the initial state, $\Sigma=\{0,1,\ldots,k-1\}$ the input alphabet, $\delta:Q\times \Sigma\rightarrow Q$ the transition function, $\Delta$ the output alphabet, and $\tau:Q\rightarrow \Delta$ the output function.
The $k$-DFAO $M$ generates a sequence $(c_n)_{n\geq 0}$ in the following way: for each non-negative integer $n$, the base-$k$ expansion of $n$ is read by $M$ from right to left starting from the initial state $q_0$, and the automaton moves from state to state according to its transition funciton $\delta$. When the end of the string is reached, the automaton halts in a state $q$, and the automaton outputs the symbol $c_n=\tau (q)$.

A necessary and sufficient condition \cite{Eilenberg1974} for a sequence to be 
$k$-automatic is that its {\it $k$-kernel}, defined as 
 $$\{ (u_{k^d n+j})_{n\geq 0} \mid d\in \mathbb{N},\, 0\leq j\leq k^d-1 \},$$
is finite. 
 If we let $\Lambda_{i}^{(k)}$ denote the operator that sends a 
 sequence $(u(n))_{n\geq 0}$ to its subsequence $(u(k n+i))_{n\geq 0}$,
 then the $k$-kernel can be defined alternatively 
as the smallest set containing $\bf{u}$ that is stable under $\Lambda_i^{(k)}$ for $0\leq i<k-1$. We write $\Lambda_i$ instead of $\Lambda_i^{(k)}$ when the value of $k$ is clear from the context.
We will use the fact that for an integer $m\geq 1$, a sequence is $k$-automatic
if and only if it is $k^m$-automatic \cite{Eilenberg1974}.

For a $k$-automatic sequence $\mathbf{u}$, we can construct a $k$-DFAO 
that generates it from its $k$-kernel. The set of states $Q$ will be in 
bijection with the $k$-kernel, so we choose to identify them. 
For $q\in Q$ and $0\leq j< k$, the 
value of the transition function $\delta(q,j)$ is defined as $\Lambda_j q$;
the output function $\tau$ maps $q$ to the $0$-th term of $q$.
This automaton has the property that leading $0$'s in the input
 does not change the output. It is minimal
among $k$-automata with this property that generates $\mathbf{u}$.

We refer the readers to  \cite{Allouche2003Sh} for a comprehensive exposition 
of automatic sequences.

In this article we will consider the Thue-Morse sequence and the period-doubling
sequence. For two distinct element $a$ and $b$ from an alphabet, the 
$(a,b)$-Thue-Morse sequence is the sequence $\mathbf{t}$ defined as 
the fixed point $s^\infty(a)$ 
of the substitution $s: a\mapsto ab,\; b\mapsto ba$;
the $(a,b)$-period-doubling sequence $\mathbf{p}$ is defined as the fixed point
$\sigma(a)$ of the substitution $\sigma: a\mapsto ab,\; b\mapsto aa$.

\subsubsection{Continued fractions}
Let $K$ be a field.  
Given a sequence of polynomials $a_j(z)\in K[z]\backslash K$,
 we may define the infinite continued fraction
\begin{equation}\label{eq:ncf}
	\CF(\mathbf{a}(z)):=\cfrac{1}{a_0(z)+\cfrac{1}{a_1(z)+\cfrac{1}{a_2(z)+\cfrac{1}{\ddots}}}}
\end{equation}
as the limit of the finite continued fractions
\begin{equation}\label{eq:fncf}
	\CF_n(\mathbf{a}(z))=\cfrac{1}{a_0(z)+\cfrac{1}{a_1(z)+\cfrac{1}{\ddots+\cfrac{1}{a_n(z)}}}}
	\in K((1/z)).
\end{equation}
The existence of the limit is guaranteed by the convergence theorem, 
whose proof is completely analogous to that for the classical continued fracions
with positive integer partial quotients.

Define the sequences $(P_n(z))$ and $(Q_n(z))$ by 
\begin{equation}\label{eq:pq}
	\begin{pmatrix} P_n(z)&Q_{n}(z)\\ P_{n-1}(z)&Q_{n-1}(z) \end{pmatrix}
:=
	\begin{pmatrix} a_n(z)&1\\ 1&0 \end{pmatrix}
		\begin{pmatrix} a_{n-1}(z)&1\\ 1&0 \end{pmatrix}
	\cdots
	\begin{pmatrix} a_0(z)&1\\ 1&0 \end{pmatrix}
	\begin{pmatrix} 0&1\\ 1&0 \end{pmatrix}
\end{equation}
for $n\geq 0$, 
then 
$$\CF_n(z;a,b)=P_n(z)/Q_n(z)\in K((1/z)),$$
for $n\geq 0$. 
Note that here rational fractions are expanded in $1/z$.  
	The unsimplified fraction $P_n(z)/Q_n(z)$ is called the
$n$-th {\it	convergent} of $\CF(x;\mathbf{a})$.

\medskip
Conversely, let
\begin{equation}
	f(z)=c_nz^n+c_{n-1}z^{n-1}+\cdots+c_0+c_{-1}z^{-1}+\cdots
\end{equation}
be an arbitrary element of $K((1/z))$.
Define the integer part of $f(z)$ as
\begin{equation}
	[f(z)]=c_nz^n+c_{n-1}z^{n-1}+\cdots+c_0.
\end{equation}

Set $f_0=f$, $a_0=[f_0]$, $f_0=a_0+1/f_1$, $a_1=[f_1]$, $f_1=a_1+1/f_2$, 
$a_2=[f_1]$...
Then $a_0\in K[z]$ and $a_j\in K[z]\backslash K$ for $j\geq 1$, and $f(z)$
admits the following continued fraction expansion 
\begin{equation}\label{eq:ncfe}
	f(z)=a_0(z)+\cfrac{1}{a_1(z)+\cfrac{1}{a_2(z)+\cfrac{1}{a_3(z)+\cfrac{1}{\ddots}}}}.
\end{equation}

\subsubsection{Stieltjes continued fractions}
Let $(u_j)_{j\geq 0}$ be a sequence taking value in $K^\times$, the Stieltjes continued
fraction  
\begin{equation}\label{eq:sfrac}
	\Stiel(x;\mathbf{u}):
	=\cfrac {u_0}{ 1+ \cfrac{u_1x}{ 1+ \cfrac{u_2x}{ 1+\cfrac {u_3x}{\ddots}}} }
\end{equation}
is defined to be the limit of the finite Stieltjes continued fractions
\begin{equation}\label{eq:tms}
	\Stiel_n(x;\mathbf{u}):
	=\cfrac {u_0}{ 1+ \cfrac{u_1x}{ 1+ \cfrac{u_2x}{ 1+\cfrac{\ddots} { 1+u_nx}}} }
	\in K[[x]].
\end{equation}
It can be easily shown that the sequence $\Stiel_n(x;\mathbf{u})$ is convergent.

Define the sequence $(P_n(x))$ and $(Q_n(x))$ by 
\begin{equation}
	\begin{pmatrix}P_n(x)&Q_n(x)\\P_{n-1}(x)&Q_{n-1}(x)\end{pmatrix}:=
\begin{pmatrix} 1&u_{n}x\\ 1&0\end{pmatrix}
\begin{pmatrix} 1&u_{n-1}x\\ 1&0\end{pmatrix}
	\cdots \begin{pmatrix} 1&u_0x \\1&0\end{pmatrix}
		\begin{pmatrix}0&1\\1/x&0\end{pmatrix}
\end{equation}
	for $n\geq 0$. Then 
	$$\Stiel_n(x;\mathbf{u})=\frac{P_n(x)}{Q_n(x)},$$
	for $n\geq 0$. 
	The unsimplified fraction $P_n(x)/Q_n(x)$ is called the
$n$-th {\it	convergent} of $\Stiel(x;\mathbf{u})$.

Unlike continued fractions, every formal power series in $K[[x]]$ can not be
expanded as a Stieltjes continued fraction.

\subsection{Conjectures}
We put forward the following conjectures concerning the Thue-Morse 
and period-doubling continued
fractions and Stieltjes continued fractions. 

\begin{Conjecture}\label{conj:tmncf}
Let $a,b$ be two distinct elements from
	$\mathbb{F}_2[z]\backslash\mathbb{F}_2$. 
	Let $\mathbf{u}(z)$ be the $(a,b)$-Thue-Morse sequence. 
	The continued fraction $\CF(\mathbf{u}(z))$
	is algebraic of degree $4$ over $\mathbb{F}_2(z)$.
\end{Conjecture}

\begin{Conjecture}\label{conj:tms}
Let $k\geq 2$ be an integer. Let $a,b$ be two distinct elements from 
$\mathbb{F}_{2^k}^\times$. Let $\mathbf{u}$ be the $(a,b)$-Thue-Morse sequence.
	The Stieltjes continued fraction $\Stiel(x;\mathbf{u})$ is algebraic
	over $\mathbb{F}_2^k(x)$. Its minimal polynomial is
		$$p_0(x)+p_1(x)y+p_2(x)y^2+p_4(x)y^4,$$
		where
		\begin{align*}
			p_0(x)&=(a^{2} b^{4} + b^{6})/{a^{4}} x^{2} + {b^{5}}/(a^{5} + a^{4} b),\\
			p_1(x)&=((ab^4 + b^5)/a^5)x + b^4/a^5,\\
			p_2(x)&=b^4/a^5x + b^4/(a^6 + a^5b),\\
			p_4(x)&=(b^4/(a^6 + a^5b))x^2.
		\end{align*}
\end{Conjecture}
Let $\mathbf{v}$ be the $(a/b,1)$-Thue-Morse sequence, then 
$$\Stiel(x;\mathbf{u})=b\cdot \Stiel(bx;\mathbf{v}).$$ 
Therefore conjecture \ref{conj:tms} admits the following equivalent form:

\begin{manualtheorem}{\ref{conj:tms}a}\label{conj:tmsa}
Let $k\geq 2$ be an integer. Let $a$ an elements from $\mathbb{F}_{2^k}^\times$ 
	distinct from $1$. Let $\mathbf{u}$ be the $(a,1)$-Thue-Morse sequence.
	The Stieltjes continued fraction $\Stiel(x;\mathbf{u})$ is algebraic
	over $\mathbb{F}_2^k(x)$. Its minimal polynomial is
		$$p_0(x)+p_1(x)y+p_2(x)y^2+p_4(x)y^4,$$
		where
		\begin{align*}
			p_0(x)&=(a^2 + 1)/a^4)x^2 + 1/(a^5 + a^4),\\
			p_1(x)&=((a + 1)/a^5)x + 1/a^5,\\
			p_2(x)&=1/a^5x + 1/(a^6 + a^5),\\
			p_4(x)&=(1/(a^6 + a^5))x^2.
		\end{align*}
\end{manualtheorem}

Or still
\begin{manualtheorem}{\ref{conj:tms}b}\label{conj:tmsb}
	We regard $a$ as a formal variable. Let $\mathbf{u}$ be the $(a,1)$-Thue-Morse
	sequence. Then 
		the Stieljtes continued fraction 
		$\Stiel(x;\mathbf{u})\in \mathbb{F}_2(a)[[x]]$ 
		is algebraic over $\mathbb{F}_2(a)(x)$.
		Its minimal polynomial is
		$$p_0(x)+p_1(x)y+p_2(x)y^2+p_4(x)y^4,$$
		where
		\begin{align*}
			p_0(x)&=(a^2 + 1)/a^4)x^2 + 1/(a^5 + a^4),\\
			p_1(x)&=((a + 1)/a^5)x + 1/a^5,\\
			p_2(x)&=1/a^5x + 1/(a^6 + a^5),\\
			p_4(x)&=(1/(a^6 + a^5))x^2.
		\end{align*}
\end{manualtheorem}

It is clear that conjecture \ref{conj:tmsb} implies conjecture \ref{conj:tmsa},
noticing that the only roots of the denominators of the coefficients of
$p_j(x)$, $j=0,1,2,4$, are $0$ and $1$.
On the other hand, if conjecture \ref{conj:tmsb} does not hold, then
$$0\neq p_0(x)+p_1(x)\Stiel(x;\mathbf{u})+p_2(x)\Stiel(x;\mathbf{u})^2+p_4(x)\Stiel(x;\mathbf{u})^4=:\sum_{n=0}^\infty
c_n(a) x^n,$$
and there exists an $n\in\mathbb{N}$ for which $c_n(a)\in \mathbb{F}_2(a)$
is not the zero. Necessarily there exists a $k\geq 2$ and an element 
$u\in \mathbb{F}_{2^k}\backslash\{0,1\}$ that is not a root of the 
numerator of $c_n(a)$, and consequently conjecture \ref{conj:tmsa} does not hold for $a=u$.
 
 \medskip
Based on our calculation, we believe that the period-doubling continued 
fractions are also algebraic. However, the period-doubling Stieltjes continued
fractions seem to be transcendental.

\begin{Conjecture}\label{conj:pdncf}
Let $a,b$ be two distinct elements from $\mathbb{F}_2[z]\backslash\mathbb{F}_2$. 
	Let $\mathbf{u}(z)$ be the $(a,b)$-period-doubling sequence. 
	The continued fraction $\CF(\mathbf{u}(z))$
	is algebraic over $\mathbb{F}_2(z)$.
\end{Conjecture}

	\begin{Conjecture}\label{conj:pds}
  Let $J\in\mathbf{F}_4\backslash \{0,1\}$.
  Let $\mathbf{u}$ be the $(1,J)$-period-doubling sequence.
  The Stieltjes continued fraction $\Stiel(x;\mathbf{u})$
  is transcendental over $\mathbb{F}_2(z)$.
\end{Conjecture}

\subsection{Main results}
  
	We developped a method for checking conjecture \ref{conj:tmncf} and
	\ref{conj:tmsa} and implemented it. 
Using this method, we checked that conjecture \ref{conj:tmncf} holds for all pairs $(a,b)$ of 
	elements from $\mathbb{F}_2[x]\backslash \mathbb{F}_2$ 
	such that $\deg a+ \deg b \leq 7$,
	and that conjecture \ref{conj:tmsa} holds for all 
	$a\in \mathbb{F}_{2^k}\backslash\{0,1\}$ for $k=2,3,4$.

	For the verification of conjectures \ref{conj:tmncf} and \ref{conj:tmsa},
	we use the Guess'n'Prove method. 
	For conjecture \ref{conj:tmncf}, our program takes
	the pair $(a,b)$ as input, and, for the $(a,b)$-Thue-Morse sequence 
	$\mathbf{u}$, uses the Derksen algorithm for Pad\'e-Hermite approximants
	\cite{Bostan2017S} to guess the minimal polynomial of $\CF(\mathbf{u})(z)$.
	To prove that the guess is correct, it then  guesses and proves several lemmas, 
	whose forms depend on the choice of $(a,b)$, that would lead to the final
	result.
	In Section \ref{section:tmncf} we illustrate our method with an example 
	of computer generated proof. 
	The proofs for the other pairs that we have tested can be found 
	on the personal web page of the authors \footnote{\;
\texttt{http://irma.math.unistra.fr/\~{}guoniu/frconj/}}.

	For conjecture \ref{conj:tmsa}, the situation is similar, except that 
	we choose to regard $a$ as a formal variable whenever we can. In this way
	we prepare a common part for all $a$, and to prove that conjecture 
	\ref{conj:tmsa} holds for a certain $a$, we only need to fill in the rest
	of the proof for this specific $a$.

	In the proofs, we exploit the structure of the automata that generate
	the algebraic series in question. For this, we need to first obtain
	a $k$-automaton of an algebraic series from an annihilating polynomial of it.
In Section \ref{sec:eq2kern} we explain this part of our program. The algorithm
is based on the proof of theorem 1 of \cite{Christol1980KMFR}. 

\medskip

Our method for checking conjecture \ref{conj:tmncf} can be adpated for	the 
verification of conjecture \ref{conj:pdncf}. We give two examples in 
Section \ref{section:pd}.

	\section{Thue-Morse Continued Fraction}\label{section:tmncf}
	Our program tests conjecture \ref{conj:tmncf} for a given
	 pair of distinct elements $(a,b)$ from $\F2[x]\backslash \F2$.
	We have checked that the conjecture holds in the case
	where $\deg a+ \deg b \leq 7$.

	The following is an example of proof that conjecture \ref{conj:tmncf}
	holds for  $(a,b)=(z, z^2 + z + 1)$. Both the statement of the theorem
	and its proof are generated automatically by our program.
	The exact statement of theorem \ref{th:ab}, lemma \ref{lem:tmncf2} and 
	\ref{lem:tmncf1} depends on the choice of $(a,b)$.

\subsection{Statement of the theorem for $(a,b)=(z, z^2 + z + 1)$}
\begin{Theorem}\label{th:ab}
	Let $(a,b)=(z, z^2 + z + 1)\in (\F2[z]\backslash \F2)^2$. Let $\mathbf{t}$
	be the $(a,b)$-Thue-Morse sequence, and $\bar{\mathbf{t}}$, the 
	$(b,a)$-Thue-Morse sequence.
	The two power series $\CF(\mathbf{t}(z))$ and $\CF(\bar{\mathbf{t}}(z))$  are 
	algebraic over ${\mathbb F}_2(z)$, with minimal polynomials of the form
  \begin{align*}
		p_4(z)y^4+p_3(z)y^3+p_2(z)y^2+p_1(z)y+p_0(z)=0.
  \end{align*}
 For $\CF(\mathbf{t}(z))$
\begin{align*}
p_{0}(z)&=z^9 +z^7 +z^6 +z^5 +z^4 +z +1 ,\\
p_{1}(z)&=z^{11} +z^{10} +z^8 +z^6 +z^5 +z^3 +z^2 +z ,\\
p_{2}(z)&=z^{12} +z^{10} +z^2 ,\\
p_{3}(z)&=z^{11} +z^{10} +z^8 +z^6 +z^5 +z^3 +z^2 +z ,\\
p_{4}(z)&=z^{10} +z^9 +z^7 +z^6 +z^5 +z^2 +z ,
\end{align*}
	and for $\CF(\bar{\mathbf{t}}(z))$
\begin{align*}
p_{0}(z)&=z^9 +z^8 +z^7 +z^6 +z^5 +z^4 +z ,\\
p_{1}(z)&=z^{11} +z^{10} +z^8 +z^6 +z^5 +z^3 +z^2 +z ,\\
p_{2}(z)&=z^{12} +z^{10} +z^2 ,\\
p_{3}(z)&=z^{11} +z^{10} +z^8 +z^6 +z^5 +z^3 +z^2 +z ,\\
p_{4}(z)&=z^{10} +z^9 +z^8 +z^7 +z^6 +z^5 +z^2 +z +1 .
\end{align*}
\end{Theorem}

\subsection{Proof}
Define
\begin{align*}
	M_n(x)&=
	x^{\deg\left(t_{2^n-1}\right)}
	\begin{pmatrix} t_{2^n-1}(1/x)&1\\1&0\end{pmatrix}
	x^{\deg\left(t_{2^n-2}\right)}
		\begin{pmatrix} t_{2^n-2}(1/x)&1\\1&0\end{pmatrix}\cdots\\
			&\quad x^{\deg\left(t_{0}\right)}
			\begin{pmatrix} t_0(1/x)&1\\1&0\end{pmatrix}
\end{align*}
and
\begin{align*}
	W_n(x)&=
	x^{\deg\left(\bar{t}_{2^n-1}\right)}
	\begin{pmatrix} \bar{t}_{2^n-1}(1/x)&1\\1&0\end{pmatrix}
	x^{\deg\left(\bar{t}_{2^n-2}\right)}
		\begin{pmatrix} \bar{t}_{2^n-2}(1/x)&1\\1&0\end{pmatrix}\cdots\\
			&\quad x^{\deg\left(\bar{t}_{0}\right)}
			\begin{pmatrix} \bar{t}_0(1/x)&1\\1&0\end{pmatrix}
\end{align*}
where $\mathbf{\bar{t}}$ is the $(b,a)$-Thue-Morse sequence. By the property
of the Thue-Morse sequence, we have for all $n\geq 0$
\begin{align*}
	M_{n+1}(x)&=W_n(x)\cdot M_n(x),  \\W_{n+1}(x)&=M_n(x)\cdot W_n(x).
\end{align*}

Define $x:=1/z$. For an non-zero polynomial $P(z)$, we define $\tilde{P}(x)$ 
to be  $P(1/x)$.  Then
\begin{equation*}
	\CF_n(\mathbf{t}(z))=
	\frac{P_n(z)}{Q_n(z)}=\frac{\tilde{P}_n(x)}{\tilde{Q}_n(x)}
	\in\mathbb{F}_2((x))=\mathbb{F}_2((1/z)).
\end{equation*}
Comparing the definition of $M_n(x)$ with definition \eqref{eq:pq}, we see that
\begin{align*}
	M_n(x)_{0,1}&= x^{d_n} \tilde{P}_{2^n-1}(x),\\
	M_n(x)_{0,0}&= x^{d_n} \tilde{Q}_{2^n-1}(x),
\end{align*}
for some positive integer $d_n$,
and
\begin{align}\label{eq:cm}
	\CF_{{2^{2n}-1}}(\mathbf{t}(z))
	=\frac{\tilde{P}_{2^{2n}-1}(x)}{\tilde{Q}_{2^{2n}-1}(x)}
	= \frac{M_{2n}(x)_{0,1}}{M_{2n}(x)_{0,0}}.
\end{align}

Our strategy is to first prove that 
both $M_{2n}(x)_{0,1}$ and $M_{2n}(x)_{0,0}$ converge to algebraic series in 
$\mathbb{F}_2[[x]]$, and then use their minimal polynomials to obtain that
of $ \CF_n(\mathbf{t}(z))$.

Actually, we will prove that for all $0\leq i,j\leq 1$ the four sequences
$(M_{2n}(x)_{i,j})_n$, $(M_{2n+1}(x)_{i,j})_n$, $(W_{2n}(x)_{i,j})_n$, and 
$(W_{2n+1}(x)_{i,j})_n$ converge to algebraic series in $\mathbb{F}_2[[x]]$.
For this purpose, we define four $2\times 2$ matrices $M^e, M^o, W^e, W^o$ as 
follows: For each $T\in\{M^e, M^o, W^e, W^o\}$ and $0\leq i,j\leq 1$,
$T_{i,j}$ is defined to be the unique solution in $\mathbb{F}_2[[x]]$ of
the polynomial $\phi(T,i,j)$ under certain initial conditions; the polynomials
$\phi(T,i,j)$ and initial conditions are given  in Subsection \ref{data:tmncf}. 
We will prove that these four matrices, whose components are algebraic
by definition, are the limits of 
$(M_{2n}(x))_n$, $(M_{2n+1}(x))_n$, $(W_{2n}(x))_n$, and $(W_{2n+1}(x))_n$.

Let us explain how the polynomials $\phi(T,i,j)$ and initial conditions 
are found, and why the solutions exist and are unique.
For $0\leq i,j\leq 1$, the the coefficients of the polynomial
$\phi(M^e,i,j)$ 
(resp. $\phi(M^o,i,j)$, $\phi(W^e,i,j)$, and $\phi(W^o,i,j)$) 
are the Pad\'e-Hermite approximants
of type 
$$(75,75,75,75,75)$$ 
of 
the vector 
$$(1, f^3, f^6, f^9, f^{12}),$$ 
where 
$f=M_{12,i,j}$ (resp. $M_{11,i,j}$, $W_{12,i,j}$, 
and $W_{11,i,j}$). See Chapter 7 of \cite{Bostan2017S} for a description
of the Derksen algorithm that is used here to find the Pad\'e-Hermite 
approximants.
We take the first eight terms of  $M_{12,i,j}$ (resp. $M_{11,i,j}$, 
$W_{12,i,j}$, and $W_{11,i,j}$) as the initial conditions for $\phi(M^e,i,j)$ 
(resp. $\phi(M^o,i,j)$, $\phi(W^e,i,j)$, and $\phi(W^o,i,j)$).
The following fact will be used to ensure that the solution exists and is 
unique:
let	$P(x,y)\in\mathbb{F}_2[x,y]$ and for each series
$f(x)=\sum_{n=0}^\infty a_n x^n\in \F2[[x]]$ denote the partial sum 
$\sum_{j=0}^{n-1}a_j x^j$ by $f_n(x)$ for $n\geq 0$. If
	for some $n\geq 0$ and $a_0, a_1, \ldots, a_{n-1}\in\F2$
	$P(x,\sum_{j=0}^{n-1}a_j x^j)=O(x^n)$
	and $Q(x,y):= P(x, \sum_{j=0}^{n-1}a_j x^j+x^ny)$ can be written as
	$x^m\sum_{j=0}^\infty q_j(x)y^j$ where $q_j(x)$ are polynomials for 
	$j\geq 0$, $q_1(0)=1$,
	and $q_j(0)=0$ for $j>1$, then there exists a unique solution 
	$f(x)\in \F2[[x]]$ of $P(x,f(x))=0$ that satisfies the initial condition
	$f_n(x)=\sum_{j=0}^{n-1}a_j x^j$.

\medskip

We state two lemmas concerning the four matrices $M^e, M^o, W^e, W^o$.
The first one is about relations between them;
the second, about the structure of the each matrix.

\begin{Lemma}\label{lem:tmncf0}
	We have
  \begin{align}	
		M^e&=W^o\cdot M^o,\label{eq:me}\\
		M^o&=W^e\cdot M^e,\label{eq:mo}\\
		W^e&=M^o\cdot W^o,\label{eq:we}\\
		W^o&=M^e\cdot W^e.\label{eq:wo}
  \end{align}	
	\end{Lemma}
\begin{proof}
	We give the proof of the identity
	$$M^e_{0,0}=W^o_{0,0}M^o_{0,0}+W^o_{0,1}M^o_{1,0},$$
	the proofs of the others being similar. 

	First, we compute the minimal polynomials of $W^o_{0,0}M^o_{0,0}$ and
	$W^o_{0,1}M^o_{1,0}$.
	We know that
	$$P(x,y)=\Res_z\left(\phi(W^o,0,0)(x,z),\  z^{12}\cdot \phi(M^o, 0,0)(x,y/z)\right)$$
	is an annihilating polynomial of $W^o_{0,0}M^o_{0,0}$; here
	$Res_z$ means the resultant with respect to the variable $z$ (
	see Chapter 6 of \cite{Bostan2017S}). 
	We use Pad\'e-Hermite approximation
	to find a candidate for the minimal polynomial of $W^o_{0,0}M^o_{0,0}$, that 
	will be called $\phi_0(x,y)$. To prove that $\phi_0(x,y)$ is
	indeed the minimal polynomial, it suffices to prove that it is an
	irreducible factor of $P(x,y)$ of multiplicity $m$ and that 
	$Q(x,y):=P(x,y)/\phi_0(x,y)^m$ is not an annihilating polynomial of 
	$W^o_{0,0}M^o_{0,0}$. We verify the first point directly. For the second
	point, we truncate $W^o_{0,0}M^o_{0,0}$ to order $270$ 
	and substitute it
	for $y$ in $Q(x,y)$. We get a series of valuation less than $270$, which 
	proves that $Q(x,y)$ is not an annihilating polynomial of $W^o_{0,0}M^o_{0,0}$.
	We find the minimal polynomial $\phi_1(x,y)$ of $W^o_{0,1}M^o_{1,0}$
	in a similar way. 
	
	Now we prove that $\phi(M^e,0,0)$ is the minimal polynomial of 
$W^o_{0,0}M^o_{0,0}+W^o_{0,1}M^o_{1,0}$.
  We know that
	$$S(x,y)=\Res_z\left(\phi_0(x,z),\; \phi_1(x,y+z)\right)$$
	is an annihilating of $W^o_{0,0}M^o_{0,0}+W^o_{0,1}M^o_{1,0}$.
	We verify that $\phi(M^e,0,0)$ is an irreducible factor of $S(x,y)$ of 
	multiplicity $\mu$, and that $Q(x,y):=S(x,y)/\phi(M^e,0,0)^\mu$ is not
	an annihilating polynomial of $W^o_{0,0}M^o_{0,0}+W^o_{0,1}M^o_{1,0}$.
  To see the last point, we truncate 	$W^o_{0,0}M^o_{0,0}+W^o_{0,1}M^o_{1,0}$
	to order $330$ and substitute it for $y$ in $Q(x,y)$. We get a series of 
	valuation less than $330$, and therefore $Q(x,y)$ is not an annihilating
	polynomial of $W^o_{0,0}M^o_{0,0}+W^o_{0,1}M^o_{1,0}$.

Finally, the first $8$ terms of $M^e_{0,0}$ and 
$W^o_{0,0}M^o_{0,0}+W^o_{0,1}M^o_{1,0}$ coincide. As these first terms
	determine a unique solution of $\phi(M^e,0,0)$, we know that the two
	series are one and the same.
\end{proof}

Define
$$R^e=\begin{pmatrix}1&0\\
0&1\end{pmatrix} \qquad\text{and}\qquad R^o=\begin{pmatrix}1&0\\
0&1\end{pmatrix}.$$

\begin{Lemma}\label{lem:tmncf2}
	For all integers $k\geq 2$ and $u=2^{2k-1}$, the following identities hold.
	\begin{align*}
M^e[:\!6u]&=M^e[:\!3u] +x^uM^e[:\!2u] +x^uM^e[:\!3u] +x^{2u}M^e[:\!2u] +x^{2u}M^e[:\!3u]\\
& \quad +x^{3u}M^e[:\!2u] +(x^{3u}+x^{4u}+x^{5u})R^e,\\
W^e[:\!6u]&=W^e[:\!3u] +x^uW^e[:\!2u] +x^uW^e[:\!3u] +x^{2u}W^e[:\!2u] +x^{2u}W^e[:\!3u]\\
& \quad +x^{3u}W^e[:\!2u] +(x^{3u}+x^{4u}+x^{5u})R^e.
\end{align*}
	For all integers $k\geq 2$ and $u=2^{2k}$,
	\begin{align*}
M^o[:\!6u]&=M^o[:\!3u] +x^uM^o[:\!2u] +x^uM^o[:\!3u] +x^{2u}M^o[:\!2u] +x^{2u}M^o[:\!3u]\\
& \quad +x^{3u}M^o[:\!2u] +(x^{3u}+x^{4u}+x^{5u})R^o,\\
W^o[:\!6u]&=W^o[:\!3u] +x^uW^o[:\!2u] +x^uW^o[:\!3u] +x^{2u}W^o[:\!2u] +x^{2u}W^o[:\!3u]\\
& \quad +x^{3u}W^o[:\!2u] +(x^{3u}+x^{4u}+x^{5u})R^o.
\end{align*}
\end{Lemma}

\begin{proof}
To prove Lemma \ref{lem:tmncf2} we first construct an automaton for each 
sequence concerned, and then transform the conditions on infinitely many
$k$'s into finitely many conditions on the states of the automaton.
In the following, we will prove that for $T=M^o_{1,0}$, for all integer 
	$k\geq 2$ and $u=2^{2k}$,
\begin{align*}
T[:\!6u]&=T[:\!3u] +x^uT[:\!2u] +x^uT[:\!3u] +x^{2u}T[:\!2u] +x^{2u}T[:\!3u] +x^{3u}T[:\!2u] .
\end{align*}
The proofs of the other $15$ cases are similar.
We break down the above identity into $3$ parts:
		\begin{align*}
0&=x^{3u}T[:\!u] +x^uT[2u\!:\!3u] +T[3u\!:\!4u],\\
0&=x^{3u}T[u\!:\!2u] +x^{2u}T[2u\!:\!3u] +T[4u\!:\!5u],\\
0&=T[5u\!:\!6u],
\end{align*}
	which can be rewritten as 
		\begin{align}
0&=T[[w]_2] +T[[10w]_2] +T[[11w]_2],\label{labelp2l3}\\
0&=T[[1w]_2] +T[[10w]_2] +T[[100w]_2],\label{labelp2l4}\\
0&=T[[101w]_2],\label{labelp2l5}
\end{align}
	for all binary word $w$ of length $2k$ and $w\neq 0^{2k}$, and
		\begin{align}
0&=T[[w]_2] +T[[10w]_2] +T[[11w]_2],\label{labelp3l3}\\
0&=T[[1w]_2] +T[[10w]_2] +T[[100w]_2],\label{labelp3l4}\\
0&=T[[101w]_2],\label{labelp3l5}
\end{align}
	for $w=0^{2k}$.

First we calculate an $2$-automaton that generates $T$ from its minimal 
	polynomial and
	its first terms. This automaton
	has $124$ states; its transition function and output function can be found in
the annex. 
	Let $A(s,w)$ denote the state reached after reading $w$ from right to
	left starting from the state $s$, and $\tau$ the output function. 
	Define $$E_{2k}=\{A(i,w) : |w|=2k, w\neq 0^{2k}\}.$$
	Identities \eqref{labelp2l3} through \eqref{labelp3l5} can be written as
		\begin{align}
0&=\tau(A(s,\epsilon)) +\tau(A(s,10)) +\tau(A(s,11)),\label{labelp4l3}\\
0&=\tau(A(s,1)) +\tau(A(s,10)) +\tau(A(s,100)),\label{labelp4l4}\\
0&=\tau(A(s,101)),\label{labelp4l5}
\end{align}
	for all $s\in E_{2k}$, and 
		\begin{align}
0&=\tau(A(s,\epsilon)) +\tau(A(s,10)) +\tau(A(s,11)),\label{labelp5l3}\\
0&=\tau(A(s,1)) +\tau(A(s,10)) +\tau(A(s,100)),\label{labelp5l4}\\
0&=\tau(A(s,101)),\label{labelp5l5}
\end{align}
	for $s=A(i,0^{2k})$.
	We find that $(A(i,0^{24}),E_{24})=(A(i,0^{16}),E_{16})$, so that we only 
	have to verify that identities \eqref{labelp4l3} through \eqref{labelp5l5} hold for 
	$2\leq k\leq 12$, which turns out to be true. 
\end{proof}

In the following lemma, we express $M_{2k}$, $M_{2k+1}$, $W_{2k}$, and
$W_{2k+1}$ in terms of $M^e$, $M^o$, $W^e$, and $W^o$.
\begin{Lemma}\label{lem:tmncf1}
	For all integer $k\geq 2$, and $u=2^{2k-1}$,
		\begin{align*}
M_{2k}&=M^e[:\!3u] +x^uM^e[:\!2u] +x^{3u}R^e,\\
W_{2k}&=W^e[:\!3u] +x^uW^e[:\!2u] +x^{3u}R^e.
\end{align*}
	For all integer $k\geq 2$, and $u=2^{2k}$,
		\begin{align*}
M_{2k+1}&=M^o[:\!3u] +x^uM^o[:\!2u] +x^{3u}R^o,\\
W_{2k+1}&=W^o[:\!3u] +x^uW^o[:\!2u] +x^{3u}R^o.
\end{align*}
\end{Lemma}
\begin{proof}
Call the four identities in Lemma \ref{lem:tmncf1} also by the name
	$M_{2k}$, $W_{2k}$, $M_{2k+1}$, and $W_{2k+1}$. 
For $n=2$, the identities can be verified directly. For $n\geq 2$,	
	we claim that
	\begin{align*}
		M_{2k} \wedge W_{2k} &\Rightarrow M_{2k+1} \wedge W_{2k+1},\\
		M_{2k+1} \wedge W_{2k+1} &\Rightarrow M_{2k+2} \wedge W_{2k+2}.
	\end{align*}
	We give the proof of 
		$$M_{2k} \wedge W_{2k} \Rightarrow M_{2k+1},$$
	the proofs of the other ones being similar. Set $u=2^{2k}$ and $v=2^{2k-1}$.
	By definition and induction hypothesis, the left side of identity 
	$M_{2k+1}$ is equal to
	\begin{align}
W_{2k}M_{2k}&=\bigl( W^e[:\!3v] +x^vW^e[:\!2v] +x^{3u}R^e \bigr)\times\bigl( M^e[:\!3v] +x^vM^e[:\!2v]\nonumber\\
& \quad +x^{3u}R^e \bigr).\label{label:lhs}
\end{align}
	Call this expression $lhs$.
	Note that both sides of identity $M_{2k+1}$ have the same term of highest
	degree $x^{6v}R^o$. Therefore we only need to prove that their difference
	is $O(x^{6v})$. Using Lemma \ref{lem:tmncf0} it can be seen that the right 
side of identity $M_{2k+1}$  is congruent, modulo $x^{6v}$, to 
\begin{equation*}
W^e[:\!6v]M^e[:\!6v] +x^{2v}W^e[:\!4v]M^e[:\!4v].\label{}
\end{equation*}
For all $n\leq 6$, replace the occurrences of  $W^e[:\!n\cdot v]$ and 
	$M^e[:\!n\cdot v]$ in the above expression 
	by the reduction modulo $x^{n\cdot v}$ of the right side of the 
	corresponding identity in Lemma \ref{lem:tmncf2} and get a new 
	expression, 
	which we call $rhs$.
	Define 
	\begin{align*}
		X&:=x^v,\\
		a_n&:=W^e[n\cdot v:(n+1)\cdot v]/X^n,\\
		b_n&:=M^e[n\cdot v:(n+1)\cdot v]/X^n,\\
		c&:=R^e.
	\end{align*}
	Using the notation introduced above, we can represent the expressions $lhs$ \eqref{label:lhs} 
	and $rhs$ as  polynomials in $\mathbb{F}_2[a_1,...,a_{6},
	b_1,...,b_{6},c][X]$. Note that it is not a problem that 
	$a_j$ commutes with $b_k$ while $W^e$ does not commute with $M^e$, because
	in the expressions concerned, the products of $W^e$-terms and $M^e$-terms
	are always in the same order.
	We let the computer do the simplification and
	check that the difference between these
	two polynomials is indeed $O(X^{6})$, which completes the proof.
\end{proof}

\begin{proof}[Proof of Theorem \ref{th:ab}]
	We prove the theorem for $\CF(\mathbf{t}(z))$; for $\CF(\bar{\mathbf{t}}(z))$,
	the proof is similar.
	By Lemma \ref{lem:tmncf1}, we have
	For all $0\leq j,k\leq 1$,
	\begin{align*}
		\lim_{n\rightarrow\infty} M_{2n,j,k}&=M^e_{j,k},\\
		\lim_{n\rightarrow\infty} M_{2n+1,j,k}&=M^o_{j,k},\\
		\lim_{n\rightarrow\infty} W_{2n,j,k}&=W^e_{j,k},\\
		\lim_{n\rightarrow\infty} W_{2n+1,j,k}&=W^o_{j,k}.
	\end{align*}
	Let $z=1/x$. By the convergence theorem and identity \eqref{eq:cm}, 
		$$\CF(\mathbf{t}(z))=\frac{M^e_{0,1}(x)}{M^e_{0,0}(x)}.$$
By definition, that $\phi(M^e,0,1)$ and
	$\phi(M^e,0,0)$ are minimal polynomials of $M^e_{0,1}$ and $M^e_{0,0}$.
	Therefore 
	$$P(x,y)=\Res_t\left(\phi(M^e,0,1)(x,t),y^{12}\phi(M^e,0,1)(x,t/y)\right)$$
	is an annihilating polynomial of $f(x)=M^e_{0,1}/M^e_{0,0}$. 
	
	Define 
	$$Q(x,y)=q_4(x)y^4+q_3(x)y^3+q_2(x)y^2+q_1(x)y+q_0(x),$$
	where
\begin{align*}
q_{0}(x)&=x^{12} +x^{11} +x^8 +x^7 +x^6 +x^5 +x^3 ,\\
q_{1}(x)&=x^{11} +x^{10} +x^9 +x^7 +x^6 +x^4 +x^2 +x ,\\
q_{2}(x)&=x^{10} +x^2 +1 ,\\
q_{3}(x)&=x^{11} +x^{10} +x^9 +x^7 +x^6 +x^4 +x^2 +x ,\\
q_{4}(x)&=x^{11} +x^{10} +x^7 +x^6 +x^5 +x^3 +x^2 .
\end{align*}
	The polynomial $Q(x,y)$ is the candidate for the minimal polynomial of $f(x)$ found by 
	Pad\'e-Hermite approximation. 
	To prove that it is indeed the minimal polynomial of $f(x)$, 
	we only need to prove that it is an irreducible factor of $P(x,y)$ of 
	multiplicity $m$ and $R(x,y):=P(x,y)/Q(x,y)^m$
	is not an annihilating polynomial of $f(x)$.
	We verify the first point directly. For the second point, we find that
	when we truncate $f(x)$ to order $96$, and substitute it for $y$ in 
	$R(z,y)$, we get a series with valuation smaller than $96$, which proves that
	$R(z,y)$ is not an annihilating polynomial of $f(x)$. 
Finally, $z^{12}Q(1/z,y)$ is the minimal polynomial of $\CF(\mathbf{t}(z))=f(1/z)$.
\end{proof}

\section{Thue-Morse Stieltjes Continued Fraction}\label{sec:s}
	Using our program, we checked that conjecture \ref{conj:tmsa} holds for all 
	$a\in \mathbb{F}_{2^k}\backslash\{0,1\}$ for $k=2,3,4$.
In this section we present our method. 

\subsection{Testing of the conjecture}\label{sub:tms1}
For  $k\geq 2$, instead of all $a$ in $\mathbb{F}_{2^k}\backslash\{0,1\}$, we
only need to test one $a$ in each of the orbits of the Frobenius morphism
$\phi: a\mapsto a^2$, because if we let $\mathbf{t}$ denote the $(a,1)$-Thue-Morse
seuqence and $\phi(\mathbf{t})$ the $(\phi(a),1)$-Thue-Morse sequence, then
$$\Stiel(x;\phi(\mathbf{t}))=\phi(\Stiel(x;\mathbf{t})),$$
and they are either both algebraic or both transcendental.

For example, $\mathbb{F}_{8}\cong \mathbb{F}_2[u]/<u^3+u+1>$
is partitioned into orbits
$$\{0\},\{1\},\;\{\bar{u},\bar{u}^2,\;\bar{u}^4\},\mbox{ and } \{\bar{u}^3,\bar{u}^6,\bar{u}^5\}.$$
Therefore for $\mathbb{F}_8$, we only have to test the conjecture for 
$a=\bar{u}$ and $a=\bar{u}^3$.
Furthermore, we only have to test those $a$ in  $\mathbb{F}_{2^k}\backslash\{0,1\}$
whose orbit contains $k$ elements, because elements whose orbit has size
$l<k$ are already tested in $\mathbb{F}_{2^l}$. For example, for
$\mathbb{F}_{16}\cong \mathbb{F}_2[u]/<u^4+u+1 >$, 
the orbit of $a=\bar{u}^5$ contains only itself and $a^2$. This means
that $a^4=a$, and therefore it is already treated in $\mathbb{F}_4$.

\subsection{Our method}
The same method for testing conjecture \ref{conj:tmncf} can be used here to
test conjecture \ref{conj:tmsa}
for  $a\in \mathbb{F}_{2^k}\backslash\{0,1\}$ ($k\geq 2$), with only slight
modifications.
As most of the following have a uniform expression for all $a$, we first
regard $a$ as a formal variable. 

As in Section \ref{section:tmncf}, we define
	\begin{equation}
		M_{n}=
\begin{pmatrix} 1&t_{2^n-1}x\\ 1&0\end{pmatrix}
\begin{pmatrix} 1&t_{2^n-2}x\\ 1&0\end{pmatrix}
		\cdots \begin{pmatrix} 1&t_0x \\1&0\end{pmatrix},
	\end{equation}
	and 
	\begin{equation}
		W_{n}=
		\begin{pmatrix} 1&\bar{t}_{2^n-1}x\\ 1&0\end{pmatrix}
			\begin{pmatrix} 1&\bar{t}_{2^n-2}x\\ 1&0\end{pmatrix}
				\cdots \begin{pmatrix} 1&\bar{t}_0x \\1&0\end{pmatrix},
	\end{equation}
	where $\mathbf{t}$ is the $(a,1)$-Thue-Morse sequence, and 
	$\mathbf{\bar{t}}$, $(1,a)$-Thue-Morse sequence.
	We have $M_{n+1}=W_n\cdot M_n$ and $W_{n+1}=M_n\cdot W_n$ for all $n$.

	We define four $2\times 2$ matrices $M^e$, $M^o$, $W^e$ and $W^o$ as follows:
	For all $T\in \{M^e, M^o,W^e,W^o\}$, and all $i,j \in \{0,1\}$, 
	$T_{i,j}$ is defined to be the unique solution in $\mathbb{F}_2(a)[[x]]$
	of the polynomial
	$\phi(T,i,j)$ under certain initial conditions. The polynomials 
	$\phi(T,i,j)$ and initial conditions can be found in the annex.
	The reason for defining these matrices and how the polynomials $\phi(T,i,j)$
	and initial conditions are found are similar to those given in Section
	\ref{section:tmncf}.

As expected, the following Lemma holds:
	\begin{Lemma}\label{lem:4relations}
  \begin{align}	
		M^e&=W^o\cdot M^o,\label{eq:mes}\\
		M^o&=W^e\cdot M^e,\label{eq:mos}\\
		W^e&=M^o\cdot W^o,\label{eq:wes}\\
		W^o&=M^e\cdot W^e.\label{eq:wos}
  \end{align}	
	\end{Lemma}
	\begin{proof}
		Similar to the proof of Lemma \ref{lem:tmncf0}.
	\end{proof}

We have the following observation concerning the structure of the four matrices.
	\begin{Observation}\label{th:obs1}
		For $k\geq 2$ and $u=2^{2k-1}$ the following identities hold:
	\begin{align}
		M^e[u\!:\!2u]&=x^u\cdot (a^{u}+1)\label{eq:obs}
		\cdot M^e[:\!u] + a^{u/2}x^{u}\cdot I_2,\\
		W^e[u\!:\!2u]&=x^u\cdot (a^{u}+1)\label{eq:obs1}
		\cdot W^e[:\!u] + a^{u/2}x^{u}\cdot I_2;
	\end{align}
		for $k\geq 2$ and $u=2^{2k}$,  
	\begin{align}
		M^o[u\!:\!2u]&=x^u\cdot (a^{u}+1)\label{eq:obs2}
		\cdot M^o[:\!u] + a^{u/2}x^{u}\cdot I_2,\\
		W^o[u\!:\!2u]&=x^u\cdot (a^{u}+1)\label{eq:obs3}
		\cdot W^o[:\!u] + a^{u/2}x^{u}\cdot I_2.
	\end{align}
	\end{Observation}

	\begin{Observation}\label{th:obs2}
		For $k\geq 2$ and $u=2^{2k-1}$,  
	\begin{align*}
		M_{2k}&=M^e[:\!u]+a^{u/2}x^{u}\cdot I_2,\\
		W_{2k}&=W^e[:\!u]+a^{u/2}x^{u}\cdot I_2;
	\end{align*}
		for $k\geq 2$ and $u=2^{2k}$,  
	\begin{align*}
		M_{2k+1}&=M^o[:\!u]+a^{u/2}x^{u}\cdot I_2,\\
		W_{2k+1}&=W^o[:\!u]+a^{u/2}x^{u}\cdot I_2.
	\end{align*}
	\end{Observation}

	\begin{Lemma}\label{lem:obs}
Observation \ref{th:obs1} implies observation \ref{th:obs2}.
	\end{Lemma}
	\begin{proof}
		Let us call the four identities in observation \ref{th:obs2} also by 
		the name $M_{2k}$, $W_{2k}$, $M_{2k+1}$ and $W_{2k+1}$.
		Suppose observation \ref{th:obs2} is true.
	We want to prove observation \ref{th:obs1} by induction. For $n=2$, the 
	identities are verified directly.
	The inductive step is
	\begin{align*}
		M_{2k} \wedge W_{2k} &\Rightarrow M_{2k+1} \wedge W_{2k+1},\\
		M_{2k+1} \wedge W_{2k+1} &\Rightarrow M_{2k+2} \wedge W_{2k+2}.
	\end{align*}

	Let us show for example how to prove 
   \begin{equation}
		M_{2k} \wedge W_{2k} \Rightarrow M_{2k+1}.
   \end{equation}
  By definition, the left side of the identity $M_{2k+1}$ is equal to 
	$$W_{2k}\cdot M_{2k},$$
	which, by induction hypothesis, is equal to 
	$$(W^e[:\!u]+a^{u/2}x^{u}I_{2})\cdot
	(M^e[:\!u]+a^{u/2}x^{u}I_{2}),$$
		where $u=2^{2k-1}$.
	Therefore only need to prove that 
	\begin{equation}\label{eq:1s}
		W^e[:\!u]\cdot M^e[:\!u]+a^{u/2}x^{u}(W^e[:\!u]+
		M^e[:\!u])-M^o[:\!2u]
	\end{equation}
	is equal to zero. 
	As the degree of the above polynomial is at most $2u-1$, 
	we only need to prove that it is 
	$O(x^{2^{2n}})$. By \eqref{eq:mos}, 
	\begin{align*}
		&\quad M^o[:\!2u]\\&\equiv W^e[:\!2u]\cdot M^e[:\!2u] \mod x^{2u}\\
		&\equiv W^e[:\!u]\cdot M^e[:\!u]+W^e[u\!:\!2u]\cdot
		M^e[:\!u] +W^e[:\!u]\cdot M^e[u\!:\!2u] \mod x^{2u}
	\end{align*}
	Therefore  \eqref{eq:1s} is congruent modulo $ x^{2u}$
 to
	\begin{align*}
		a^{u/2}x^{u}(W^e[:\!u]+ M^e[:\!u])+
W^e[u:2u]\cdot M^e[:\!u] +W^e[:\!u]\cdot M^e[u\!:\!2u].
	\end{align*}
		Substitute $W^e[u\!:\!2u]$ and $M^e[u\!:\!2u]$ by the 
		expressions in observation \ref{th:obs1} and we obtain that the quantity
		above is $O(x^{2u})$. That is, expression \eqref{eq:1s} is congruent
		to $0$ modulo $x^{2u}$; since it has no term of order higher than $2u-1$,
		it is equal to~$0$.
	\end{proof}
	\begin{Proposition}
	Observation \ref{th:obs2} implies conjecture \ref{conj:tmsb}.
	\end{Proposition}
	\begin{proof}
 First,	taking the limit of the identities in observation \ref{th:obs2}, 
we have	for all $j,k\in \{0,1\}$, 
		\begin{align*}
			\lim_{n\rightarrow \infty} M_{2n,j,k}&=M^e_{j,k},\\
			\lim_{n\rightarrow \infty} M_{2n+1,j,k}&=M^o_{j,k},\\
			\lim_{n\rightarrow \infty} W_{2n,j,k}&=W^e_{j,k},\\
			\lim_{n\rightarrow \infty} W_{2n+1,j,k}&=W^o_{j,k}.
		\end{align*}
		Therefore 
		\begin{equation*}
			\Stiel(x;\mathbf{t})=\lim_{n\rightarrow \infty} \frac{P_n}{Q_n}\\
			=\lim_{n\rightarrow \infty} \frac{P_{2^{2n}-1}}{Q_{2^{2n}-1}}\\
			=\lim_{n\rightarrow \infty} \frac{M_{2n,0,1}/x}{M_{2n,0,0}}\\
			=\frac{M^e_{0,1}/x}{M^e_{0,0}}.
		\end{equation*}
		We obtain the minimal polynomial of $\Stiel(x;\mathbf{t})$
		from those of $M^e_{0,1}$ and
		$M^e_{0,0}$, using the same method described in the proof of Theorem 
		\ref{th:ab}.
	\end{proof}

\begin{Remark}
	The above proposition says that observation \ref{th:obs2} implies 
	conjecture \ref{conj:tms} when $a$ is regarded as a formal variable.
	The implication also holds when $a$ specializes as an element in 
	$\mathbb{F}_{2^k}\backslash\{0,1\}$ ($k\geq 2$).
\end{Remark}
Therefore, to prove that conjecture \ref{conj:tmsa} holds for a certain 
	$a\in\mathbb{F}_{2^k}\backslash\{0,1\}$ ($k\geq 2$), we only need to prove
	that observation \ref{th:obs1} holds for $a$.
Because of the following
argument, we only have to check \eqref{eq:obs} through \eqref{eq:obs3} for finitely
many $k$'s instead of for all $k\geq 2$:

	For $k\geq 2$ and $u=2^{2k-1}$, identity \eqref{eq:obs} and \eqref{eq:obs1} can be written as
	\begin{equation}\label{eq:w}
		T[[1w]_2]=(a^{u}+1)\cdot T[[w]_2]
	\end{equation}
	for every component $T$ of $M^e$ and $W^e$
	and all binary words $w$ of length
	$2k-1$ and $w\neq 0^{2k-1}$; and
	\begin{equation}\label{eq:0}
		T[[1w]_2]=(a^{u}+1)\cdot T[[w]_2]+a^{u/2}
	\end{equation}
	for $w=0^{2k-1}$.

	We calculate the an automaton of $T$ from the algebraic equation that 
	defines it, following the method in \cite{Christol1980KMFR} (see Section
	\ref{sec:eq2kern}).
	Let $A(s,w)$ denote the state reached after reading $w$ from right to
	left starting from the state $s$, and $\tau$ the output function. 
	Define 
	$$E_{2k-1}=\{A(i,w) \mid |w|=2k-1, w\neq 0^{2k-1}\}.$$
	Identity \eqref{eq:w} and \eqref{eq:0} can be written as
	\begin{equation}\label{eq:aw}
	\tau(A(s,1))=(a^{u}+1)\cdot \tau(A(s,\epsilon))
	\end{equation}
	for all $s\in E_{2k-1}$, and 
	\begin{equation}\label{eq:a0}
		\tau(A(s,1))=(a^{u}+1)\cdot \tau(A(s,\epsilon))+a^{u/2}
	\end{equation}
	for $s=A(i,0^{2k-1})$.

	As $E_{2k+1}$ is completely determined by $E_{2k-1}$, the sequence
	$(E_{2k+1})_k$ is ultimately periodic. The sequences $(A(i,0^{2k-1}))_k$
	and $a^{2^{2k-1}}$ are also periodic. 
	Therefore we only need to check \eqref{eq:aw} and \eqref{eq:a0} 
	for finitely many $k$'s.

\subsection{An example} 
	For $a\in \mathbb{F}_{4}\backslash \{0,1\}$ and 
	 $T=Me_{0,0}$, we find that the minimal $2$-DFAO of $T$ has as transition
	function 
$(n,j)\mapsto \delta(n,j)$ ($\Lambda(n):=[\delta(n,0),\delta(n,1)]$):
{
\renewcommand{\arraystretch}{1.2}

\begin{longtable}{| c c  |  c c  |  c c  |  c c  |}
\hline
$n$ & $\Lambda(n)$ & $n$ & $\Lambda(n)$ & $n$ & $\Lambda(n)$ & $n$ & $\Lambda(n)$ \\
\hline
0 & [1, 2]& 5 & [2, 8]& 10 & [7, 8]& 15 & [13, 17]\\
1 & [3, 4]& 6 & [9, 4]& 11 & [8, 13]& 16 & [18, 4]\\
2 & [5, 6]& 7 & [10, 4]& 12 & [14, 4]& 17 & [19, 12]\\
3 & [1, 7]& 8 & [11, 6]& 13 & [15, 16]& 18 & [16, 6]\\
4 & [4, 4]& 9 & [6, 12]& 14 & [12, 16]& 19 & [17, 8]\\
\hline
\end{longtable}

}

 and output function $n\mapsto \tau(n)$:
	{
\renewcommand{\arraystretch}{1.2}
$$
\begin{array}{| c c  |  c c  |  c c  |  c c  |  c c  |  c c  |  c c  |}
\hline
n & \tau(n) & n & \tau(n) & n & \tau(n) & n & \tau(n) & n & \tau(n) & n & \tau(n) & n & \tau(n) \\
\hline
0 & 1& 3 & 1& 6 & a + 1& 9 & a + 1& 12 & 1& 15 & 1& 18 & a\\
1 & 1& 4 & 0& 7 & 0& 10 & 0& 13 & 1& 16 & a& 19 & a + 1\\
2 & a& 5 & a& 8 & a& 11 & a& 14 & 1& 17 & a + 1&  &\\
\hline
\end{array}
$$
}

 The tuple $(A(i,0^{2k-1})$, $E_{2k-1})$ has the following values:
\begin{align*}
	k&=3: ({1}, \{2, 4, 7, 8, 9\}),\\
	k&=5: ({1}, \{2, 4, 7, 8, 9, 13, 14\}),\\
	k&=7: ({1}, \{2, 4, 7, 8, 9, 13, 14, 17, 18\}),\\
	k&=9: ({1}, \{2, 4, 7, 8, 9, 13, 14, 17, 18\}).
\end{align*}

 For all $k\geq 1$, $a^{2^{2k-1}}=a+1$.  Therefore we only have to check that
 identity \eqref{eq:obs} holds for $k=3,5,7$, which turns out to be true.

 \section{Period-doubling Continued Fractions}\label{section:pd}
 The method for checking conjecture \ref{conj:tmncf} can be adapted for the 
 verification of conjecture \ref{conj:pdncf}. In this section, we give
 two examples. First we introduce the notation.
 
	For $(a,b)\in (\F2[z]\backslash \F2)^2$. Let $\mathbf{p}$ be the
	$(a,b)$-period-doubling sequence. 
	Define two sequence $A_n(x)$ and $B_n(x)$ by 
	\begin{align*}
		A_0(x)&=x^{\deg(a)}\begin{pmatrix}
			a(1/x)&1\\
			1&0
		\end{pmatrix}	\\
		B_0(x)&=x^{\deg(b)}\begin{pmatrix}
			b(1/x)&1\\
			1&0
		\end{pmatrix}	\\
		A_{n+1}(x)&=B_n(x)A_n(x)\forall n\geq 0\\
		B_{n+1}(x)&=A_n(x)A_n(x)\forall n\geq 0\\
	\end{align*}
Define $x:=1/z$. For an non-zero polynomial $P(z)$, we define $\tilde{P}(x)$ 
to be  $P(1/x)$.  Then
\begin{equation*}
	\CF_n(\mathbf{p}(z))=
	\frac{P_n(z)}{Q_n(z)}=\frac{\tilde{P}_n(x)}{\tilde{Q}_n(x)}
	\in\mathbb{F}_2((x))=\mathbb{F}_2((1/z)).
\end{equation*}
Comparing the definition of $A_n(x)$ with definition \eqref{eq:pq}, we see that
\begin{align*}
	A_n(x)_{0,1}&= x^{d_n} \tilde{P}_{2^n-1}(x),\\
	A_n(x)_{0,0}&= x^{d_n} \tilde{Q}_{2^n-1}(x),
\end{align*}
for some positive integer $d_n$,
and
\begin{align}\label{eq:cm}
	\CF_{{2^{2n}-1}}(\mathbf{p}(z))
	=\frac{\tilde{P}_{2^{2n}-1}(x)}{\tilde{Q}_{2^{2n}-1}(x)}
	= \frac{A_{2n}(x)_{0,1}}{A_{2n}(x)_{0,0}}.
\end{align}

\subsection{ The $(z^2,z)$-period-doubling sequence}\label{subsection:pd1}
In this subsection, we prove the following theorem.
\begin{Theorem}\label{thm:pd1}
	Let $(a,b)=(z^2,z)\in (\F2[z]\backslash \F2)^2$. Let $\mathbf{p}$ be the
	$(a,b)$-period-doubling sequence. The power series $\CF(\mathbf{p}(z))$ is algebraic
	over $\F2(z)$; its minimal polynomial is
	$$
z^4+x^3z^2+	(x^5+x^4)z+x^3+x^2+1=0.
	$$
\end{Theorem}
	We define four $2\times 2$ matrices $A^e$, $A^o$, $B^e$ and $B^o$ as follows:
	For all $T\in \{A^e, A^o,B^e,B^o\}$, and all $i,j \in \{0,1\}$, 
	$T_{i,j}$ is defined to be the unique solution in $\mathbb{F}_2(a)[[x]]$
	of the polynomial
	$\phi(T,i,j)$ under certain initial conditions. The polynomials 
	$\phi(T,i,j)$ and initial conditions can be found in the annex.
	The reason for defining these matrices and how the polynomials $\phi(T,i,j)$
	and initial conditions are found are similar to those given in Section
	\ref{section:tmncf}.
\begin{Lemma}\label{lem:0}
	The following identities hold:
	\begin{align*}
		A^e&=B^o\cdot A^o,\\
		A^o&=B^e\cdot A^e,\\
		B^e&=A^o\cdot A^o,\\
		B^o&=A^e\cdot A^e.
	\end{align*}
\end{Lemma}
\begin{proof}
		Similar to the proof of Lemma \ref{lem:tmncf0}.
\end{proof}
\begin{Lemma}\label{lem:1}
For $n\geq 2$ even and $u=(5\cdot 2^n+1)/3$,
\begin{align*}
  A^e[:2u]&=A^e[:u]+x^u\cdot I_2\\
  B^e[:2u]&=B^e[:u]+x^u\begin{pmatrix}x^{u-1}&x^{u-2}+x^{u-3}\\0&x^{u-1}
  \end{pmatrix}
\end{align*}
For $n\geq 1$ odd, for $u=(5\cdot 2^n+2)/3$,
\begin{align*}
  A^o[:2u-1]&=A^o[:u]+x^{2u-2}\cdot I_2\\
  B^o[:2u-1]&=B^o[:u]+x^u\begin{pmatrix}1&x^{u-2}+x^{u-3}\\0&1
  \end{pmatrix}
\end{align*}
\end{Lemma}
\begin{proof}
	We give proof of 
	\begin{equation}\label{eq:pdlem1}
	A^e_{0,0}[:2u]=A^e_{0,0}[:u]+x^u
	\end{equation}
	for $n\geq 2$ even and $u=(5\cdot 2^n+1)/3$; the proofs of the other $15$ 
	cases are similar. Let $T=A^e_{0,0}$. Identity \eqref{eq:pdlem1} can be
	written as
	\begin{equation}\label{eq:pdlem1p}
		T[u:2u]=x^u.
	\end{equation}
	From the minimal polynomial and the first terms of $T$, we find its minimal
	automaton. Its transition function $\delta$ 
($\Lambda(n):=[\delta(n,0),\delta(n,1)]$)	
	and output function $\tau$
	are as follows:

\begin{longtable}{| c c  |  c c  |  c c  |  c c  |}
\hline
$n$ & $\Lambda(n)$ & $n$ & $\Lambda(n)$ & $n$ & $\Lambda(n)$ & $n$ & $\Lambda(n)$ \\
\hline
0 & [1, 2]& 2 & [4, 5]& 4 & [4, 4]& 6 & [6, 4]\\
1 & [3, 2]& 3 & [6, 6]& 5 & [2, 6]&  & \\
\hline
\end{longtable}

\begin{longtable}{| c c  |  c c  |  c c  |  c c  |}
\hline
$n$ & $\tau(n)$ & $n$ & $\tau(n)$ & $n$ & $\tau(n)$ & $n$ & $\tau(n)$ \\
\hline
0 & 1& 2 & 0& 4 & 0& 6 & 1\\
1 & 1& 3 & 1& 5 & 0&  & \\
\hline
\end{longtable}
	Let $A(s,w)$ denote the state reached after reading $w$ from right
	to left starting from the state $s$.

For $n\geq 2$ even, $k=n/2-1$, and $u=(5\cdot 2^n+1)/3$
	the binary expansions of integers $j$ in $[u, 2u[$ have the following forms:

\begin{longtable}{|  c  |   c   |}
	\hline
  $j$ & $[j]_2$ \\
\hline
	$j=u$ & $1(10)^k11$  \\
	$u<j<2^{n+1}$ & $ 1(10)^{k-l}11\{0,1\}^{2l}$, $1\leq l \leq k$ \\
	$2^{n+1}\leq j<2u-2$ & $ 1(10)^{l}0\{0,1\}^{n-2l}$, $0\leq l \leq k$ \\
	$j=u-1, u-2$ & $1(10)^k10\{0,1\}$\\
\hline
\end{longtable}

Consider the sets
	\begin{align*}
		B_0&=\{1(10)^k11 \mid k\geq 0 \}\\
 	B_1&=\{ 1(10)^{m}11\{0,1\}^{2l} \mid l\geq 1,\;m\geq 0\}\\
 	B_2&=\{ 1(10)^{m}0\{0,1\}^{2l} \mid l\geq 1,\;m\geq 0\}\\
 	B_3&=\{ 1(10)^m10\{0,1\}\mid m\geq 0\}.
	\end{align*}
For $i=0,1,2,3$, define
	$$E_i=\{ A(0,w) \mid w\in B_i \}.$$
We find that
	\begin{align*}
		E_0&=\{6\} \\
		E_1&= \{4\}\\
		E_2&= \{4, 5\}\\
		E_3&= \{4\}\\
	\end{align*}
We verify that for all $s\in E_0$, $\tau(s)=1$, and for all $s\in E_i$, $i=1,2,3$, $\tau(s)=0$. This proves identity \eqref{eq:pdlem1p} for all $n\geq 2$ even,
and $u=(5\cdot 2^n+1)/3$.
\end{proof}

\begin{Lemma}\label{lem:2}
For $n\geq 2$ even and $u=(5\cdot 2^n+1)/3$
\begin{align*}
  A_{n}&=A^e[:u]+x^u\cdot I_2,\\
  B_{n}&=B^e[:u].
\end{align*}
For $n\geq 1$ odd, for $u=(5\cdot 2^n+2)/3$
\begin{align*}
  A_{n}&=A^o[:u],\\
  B_{n}&=B^o[:u]+x^u\cdot I_2.\\
\end{align*}
\end{Lemma}
\begin{proof}
	Let us call the identities involving $A_n$ and $B_n$ also by the name
	$A_n$ and $B_n$. We will prove the lemma by induction.
	It can be verified directly that $A_n$ and $B_n$ holds for $n=1,\; 2$. 
 For the inductive step, we want to prove that for $n\geq 2$, 
	$$
		A_n \wedge B_n \Rightarrow A_{n+1} \wedge B_{n+1}.
		$$
		We give the proof of 
	$$
		A_{n} \wedge B_n \Rightarrow A_{n+1}
		$$
		when $n$ is even. The proofs of the other cases are similar.
		Now suppose that for some $n\geq 2$ even and $u=(5\cdot 2^n+1)/3$ 
		it holds that
\begin{align*}
	A_{n}&=A^e[:u]+x^u\cdot I_2, \\
	B_{n}&=B^e[:u].
\end{align*}
We want to prove that
	\begin{equation}\label{eq:ao}
  A_{n+1}=A^o[:2u].
	\end{equation}
	By definition and induction hypothesis, 
\begin{align*}
	A_{n+1}&= B_{n}A_{n},\\
	&=B^e[:u]\cdot (A^e[:u]+x^u\cdot I_2).
\end{align*}
	As the degrees of both $A_{n+1}$ and $A^o[:2u]$ are at most $2u-1$, to prove 
	that they are equal, we only need to prove that they are congruent modulo
	$x^{2u}$.
	By lemma \ref{lem:0} and lemma \ref{lem:1},
\begin{align*}
	A^o[:2u]&\equiv B^e[:2u]\cdot A^e[:2u]\\
	&\equiv 
\left	(B^e[:u]+x^u\begin{pmatrix}x^{u-1}&x^{u-2}+x^{u-3}\\0&x^{u-1}
	\end{pmatrix}\right)\cdot
	(A^e[:u]+x^u\cdot I_2) \mod x^{2u}.
\end{align*}
Therefore
\begin{align*}
	A^o[:2u]-A_{n+1}&\equiv x^u\begin{pmatrix}x^{u-1}&x^{u-2}+x^{u-3}\\0&x^{u-1} 
	\end{pmatrix} \cdot A^e[:u]\\
	&\equiv x^u\begin{pmatrix}x^{u-1}&x^{u-2}+x^{u-3}\\0&x^{u-1} 
	\end{pmatrix} \cdot (A^e[:u]\mod x^3)\\
	&\equiv x^u\begin{pmatrix}x^{u-1}&x^{u-2}+x^{u-3}\\0&x^{u-1} 
	\end{pmatrix} \cdot 
	\begin{pmatrix}
		1&x^2\\x^2&0
	\end{pmatrix}\\
	&\equiv 0 \mod x^{2u}.\qedhere
\end{align*}
\end{proof}
	Theorem \ref{thm:pd1} can be derived from lemma \ref{lem:2} and the definition
	of $A^e$, using the same method as in the proof of theorem \ref{th:ab}.

\subsection{ The $(z^3, z^2+z+1)$-period-doubling sequence}\label{subsection:pd2}
In this subsection, we prove the following theorem.
\begin{Theorem}\label{thm:pd2}
	Let $(a,b)=(z^3,z^2+z+1)\in (\F2[z]\backslash \F2)^2$. Let $\mathbf{p}$ be the
	$(a,b)$-period-doubling sequence. The power series $\CF(\mathbf{p}(z))$ is algebraic
	over $\F2(z)$; its minimal polynomial is
	$$
	z^4+ (x^5 + x^4 + x^3)z^2+(x^8 + x^6 + x^5 + x^3)z+x^5 + x^3 + x^2=0
	$$
\end{Theorem}
	We define four $2\times 2$ matrices $A^e$, $A^o$, $B^e$ and $B^o$ as follows:
	For all $T\in \{A^e, A^o,B^e,B^o\}$, and all $i,j \in \{0,1\}$, 
	$T_{i,j}$ is defined to be the unique solution in $\mathbb{F}_2(a)[[x]]$
	of the polynomial
	$\phi(T,i,j)$ under certain initial conditions. The polynomials 
	$\phi(T,i,j)$ and initial conditions can be found in the annex.
	The reason for defining these matrices and how the polynomials $\phi(T,i,j)$
	and initial conditions are found are similar to those given in Section
	\ref{section:tmncf}.
\begin{Lemma}\label{lem:00}
	The following identities hold:
	\begin{align*}
		A^e&=B^o\cdot A^o,\\
		A^o&=B^e\cdot A^e,\\
		B^e&=A^o\cdot A^o,\\
		B^o&=A^e\cdot A^e.
	\end{align*}
\end{Lemma}
\begin{proof}
		Similar to the proof of Lemma \ref{lem:tmncf0}.
\end{proof}
\begin{Lemma}\label{lem:4}
	For $n\geq 2$ even,  $u=(2^{n+3}+1)/3$, $v=2^n$, 
	\begin{align*}
		A^e[:2u]&=(1+x^v+x^{2v})\cdot (A^e[:u]+x^v A^e[:u-v])+x^{4v}A^e[:2u-4v]+\\
		&\quad  (x^u+x^{u+v}+x^{u+2v})\begin{pmatrix}1&1\\0&1\end{pmatrix},\\
		B^e[:2u]&=(1+x^v+x^{2v})\cdot (B^e[:u]+x^v B^e[:u-v])+x^{4v}B^e[:2u-4v]+\\
		&\quad  \begin{pmatrix}x^{2u-1}&x^{2u-1}+x^{2u-4}\\0&x^{2u-1}\end{pmatrix}.
	\end{align*}
	For $n\geq 1$ odd,  $u=(2^{n+3}+2)/3$, $v=2^n$, 
	\begin{align*}
		A^o[:2u-1]&=(1+x^v+x^{2v})\cdot (A^o[:u]+x^v A^o[:u-v])+x^{4v}A^o[:2u-4v-1]+\\
		&\quad  (x^{2u-2})\begin{pmatrix}1&0\\0&1\end{pmatrix},\\
		B^o[:2u-1]&=(1+x^v+x^{2v})\cdot (B^o[:u]+x^v B^o[:u-v])+x^{4v}B^o[:2u-4v-1]+\\
		&\quad  \begin{pmatrix}x^{u}+x^{u+v}+x^{u+2v}& x^{2u-2}+x^{2u-4}
		\\0& x^{u}+x^{u+v}+x^{u+2v}\end{pmatrix}.
	\end{align*}
\end{Lemma}
\begin{proof}
	We give the proof of the $(0,0)$-th component of the first identity. The
	proofs of the other $15$ identities are similar.
	Let $T=A^e_{0,0}$. We want to prove that
	\begin{equation}\label{eq:lem4}
		T[:\!2u]=(1+x^v+x^{2v})\cdot (T[:\!u]+x^v T[:\!u-v])+x^{4v}T[:\!2u-4v]+
		x^u+x^{u+v}+x^{u+2v}.
	\end{equation}
	for all $n\geq 2$ even,  $u=(2^{n+3}+1)/3$, and $v=2^n$. 
	Equation \eqref{eq:lem4} can be rewritten as
	\begin{align*}
		&\quad		T[u\!:\!2u]\\
	&=x^vT[u-v\!:\!u]+x^{2v}T[u-v\!:\!u]+x^{3v}T[:\!u-v]+\\
		&\quad	x^{4v}T[:\!2u-4v]+ x^u+x^{u+v}+x^{u+2v}\\
		&=(x^u+x^vT[u-v\!:\!2v])+ (x^vT[2v\!:\!u]+x^{3v}T[:\!u-2v])+\\
		&\quad (x^{u+v}+x^{3v}T[u-2v\!:v]+ x^{2v}T[u-v:2v])+\\
		&\quad (x^{2v}T[2v\!:\!u]+x^{3v}T[v\!:\!u-v]+x^{4v}T[\!:\!u-2v])+\\
		&\quad (x^{u+2v}+x^{4v}T[u-2v\!:\!2u-4v])
	\end{align*}
	noting that $u<3v<u+v<4v<u+2v<2u$. The above identity can be decomposed into
	five parts:
	\begin{align*}
		T[u:3v]&=x^u+x^vT[u-v\!:\!2v] \\
		T[3v:u+v]&=x^vT[2v\!:\!u]+x^{3v}T[:\!u-2v]\\
		T[u+v:4v]&=x^{u+v}+x^{3v}T[u-2v\!:v]+ x^{2v}T[u-v:2v]\\
		T[4v:u+2v]&=x^{2v}T[2v\!:\!u]+x^{3v}T[v\!:\!u-v]+x^{4v}T[\!:\!u-2v]\\
		T[u+2v:2u]&=x^{u+2v}+x^{4v}T[u-2v\!:\!2u-4v]
	\end{align*}
	That the above five identities hold for all $n\geq 2$, $u=(2^{n+3}+1)/3$,
	and $v=2^n$ is equivalent to the following identities:
	\begin{align}
		\label{eq:tf}T[[10w]_2]&=1+T[[1w]_2] &\forall w\in L_0\\
		T[[10w]_2]&=T[[1w]_2] &\forall w\in L_1\\
		T[[11w]_2]&=T[[10w]_2]+T[[w]_2] &\forall w\in L_2\\
		T[[11w]_2]&=1+T[[w]_2]+T[[1w]_2] &\forall w\in L_0\\
		T[[11w]_2]&=T[[w]_2]+T[[1w]_2] &\forall w\in L_1\\
		T[[100w]_2]&=T[[10w]_2]+T[[1w]_2]+T[[w]_2] &\forall w\in L_2\\
		T[[100w]_2]&=1+T[[w]_2]&\forall w \in L_0\\
		T[[100w]_2]&=T[[w]_2]&\forall w\in L_1\\
		T[[10w]_2]&=T[[w]_2]&\forall w\in L_3\\
		\label{eq:tl}	T[[10w]_2]&=T[[w]_2]&\forall w\in L_4
	\end{align}
	where
	\begin{align*}
		L_0&=L((10)^*11),\\
		L_1&=L((10)^*11\{00,01,10,11\}^+),\\
		L_2&=L((10)^*0\{0,1\} \{00,01,10,11\}^*+ (10)^+),\\
		L_3&=L((10)^+0\{00,01,10,11\}^+),\\
		L_4&=L((10)^+\{0,1\}).
	\end{align*}
	From the minimal polynomial and the first terms of $T$, we find its minimal
	automaton. Its transition function $\delta$ 
($\Lambda(n):=[\delta(n,0),\delta(n,1)]$)	
	and output function $\tau$
	are as follows:

{
\renewcommand{\arraystretch}{1.2}

\begin{longtable}{| c c  |  c c  |  c c  |  c c  |  c c  |}
\hline
$n$ & $\Lambda(n)$ & $n$ & $\Lambda(n)$ & $n$ & $\Lambda(n)$ & $n$ & $\Lambda(n)$ & $n$ & $\Lambda(n)$ \\
\hline
0 & [1, 2]& 6 & [11, 12]& 12 & [15, 20]& 18 & [23, 12]& 24 & [27, 15]\\
1 & [3, 4]& 7 & [13, 14]& 13 & [9, 21]& 19 & [24, 23]& 25 & [15, 18]\\
2 & [5, 6]& 8 & [5, 15]& 14 & [14, 14]& 20 & [14, 26]& 26 & [17, 20]\\
3 & [7, 8]& 9 & [16, 14]& 15 & [22, 10]& 21 & [14, 16]& 27 & [28, 14]\\
4 & [9, 10]& 10 & [17, 18]& 16 & [23, 9]& 22 & [5, 17]& 28 & [27, 17]\\
5 & [8, 9]& 11 & [19, 6]& 17 & [24, 25]& 23 & [16, 16]&  & \\
\hline
\end{longtable}

}

{
\renewcommand{\arraystretch}{1.2}
$$
\begin{array}{| c c  |  c c  |  c c  |  c c  |  c c  |  c c  |}
\hline
n & \tau(n) & n & \tau(n) & n & \tau(n) & n & \tau(n) & n & \tau(n) & n & \tau(n) \\
\hline
0 & 1& 5 & 1& 10 & 0& 15 & 1& 20 & 0& 25 & 1\\
1 & 1& 6 & 0& 11 & 0& 16 & 1& 21 & 0& 26 & 0\\
2 & 1& 7 & 1& 12 & 1& 17 & 0& 22 & 1& 27 & 0\\
3 & 1& 8 & 1& 13 & 1& 18 & 1& 23 & 1& 28 & 0\\
4 & 1& 9 & 1& 14 & 0& 19 & 0& 24 & 0&  & \\
\hline
\end{array}
$$
}

	Let $A(s,w)$ denote the state reached after reading $w$ from right
	to left starting from the state $s$.
	For $j=0,1,\ldots, 4$, define
	$$E_j=\{ A(0,w)\mid w\in L_j\}.$$
	We can compute $E_j$ explicitly and find 
	\begin{align*}
		E_0&=\{6\} \\
		E_1&=\{14,15,16,17,18,20\} \\
		E_2&=\{3, 4, 5, 13, 14, 15, 16, 17, 19, 27\} \\
		E_3&= \{9, 14, 21, 23\}\\
		E_4&= \{8,9\}.
	\end{align*}
  Equations \eqref{eq:tf} through \eqref{eq:tl} can be written as
	\begin{align}
		\label{eq:tauf}	\tau(A(s,10))&=1+\tau(A(s,1)) &\forall s\in E_0\\
		\tau(A(s,10))&=\tau(A(s,1)) &\forall s\in E_1\\
		\tau(A(s,11))&=\tau(A(s,10))+\tau(s) &\forall s\in E_2\\
		\tau(A(s,11))&=1+\tau(s)+\tau(A(s,1)) &\forall s\in E_0\\
		\tau(A(s,11))&=\tau(s)+\tau(A(s,1)) &\forall s\in E_1\\
		\tau(A(s,100))&=\tau(A(s,10))+\tau(A(s,1))+\tau(s) &\forall s\in E_2\\
		\tau(A(s,100))&=1+\tau(s)&\forall s \in E_0\\
		\tau(A(s,100))&=\tau(s)&\forall s\in E_1\\
		\tau(A(s,10))&=\tau(s)&\forall s\in E_3\\
		\label{eq:taul}		\tau(A(s,10))&=\tau(s)&\forall s\in E_4
	\end{align}
	We verify equations \eqref{eq:tauf} through \eqref{eq:taul} directly.
\end{proof}
\begin{Lemma}\label{lem:3}
	For $n\geq 2$ even, $u=(2^{n+3} + 1)/3$, $ v=2^n$, 
	\begin{align*}
	A_{n}&=A^e[:u]+x^v\cdot A^e[:u-v]+x^u\cdot \begin{pmatrix}1&1\\0&1\end{pmatrix}\\
	B_{n}&=B^e[:u]+x^v\cdot B^e[:u-v]
	\end{align*}
	For $n\geq 1$ odd, $u=(2^{n+3} + 2)/3$, $ v=2^n$, 
	\begin{align*}
	A_{n}&=A^o[:u]+x^v\cdot A^o[:u-v]\\
	B_{n}&=B^o[:u]+x^v\cdot B^o[:u-v]+x^u\cdot I_2
	\end{align*}
\end{Lemma}
\begin{proof}
	For $n\geq 2$ even, we prove that 
	$$A_n \wedge B_n \Rightarrow A_{n+1}.$$
	Set $u=(2^{n+3} + 1)/3$, $ v=2^n$. Identity $A_{n+1}$ can be written as
	\begin{equation}\label{eq:anp1}
		A_{n+1}=A^o[:2u]+x^{2v}\cdot A^o[2u-2v]
	\end{equation}
	The left side of Eq. \eqref{eq:anp1} is
	\begin{align*}
		&\quad B_n A_n\\
		&=(B^e[:u]+x^v\cdot B^e[:u-v])\left(A^e[:u]+x^v\cdot A^e[:u-v]+x^u\cdot \begin{pmatrix}1&1\\0&1\end{pmatrix}\right)\\
		&=(B^e[:u]+x^v\cdot B^e[:u-v])(A^e[:u]+x^v\cdot A^e[:u-v])\\
			&\quad x^uB^e[:u]\begin{pmatrix}1&1\\0&1\end{pmatrix}+x^{u+v}B^e[:u-v]\begin{pmatrix}1&1\\0&1\end{pmatrix}
	\end{align*}
	By lemma \ref{lem:00}, the right side of Eq. \eqref{eq:anp1} is congruent, modulo $x^{2u}$ to
		$$(1+x^{2v})A^o[:2u]= (1+x^{2v})B^e[:2u]A^e[:2u].$$
		We recall that
	\begin{align*}
		A^e[:2u]&=(1+x^v+x^{2v})\cdot (A^e[:u]+x^v A^e[:u-v])+x^{4v}A^e[:2u-4v]+\\
		&\quad  (x^u+x^{u+v}+x^{u+2v})\begin{pmatrix}1&1\\0&1\end{pmatrix},\\
		B^e[:2u]&=(1+x^v+x^{2v})\cdot (B^e[:u]+x^v B^e[:u-v])+x^{4v}B^e[:2u-4v]+\\
		&\quad  \begin{pmatrix}x^{2u-1}&x^{2u-1}+x^{2u-4}\\0&x^{2u-1}\end{pmatrix}.
	\end{align*}
	Noticing that
		$$ \begin{pmatrix}x^{2u-1}&x^{2u-1}+x^{2u-4}\\0&x^{2u-1}\end{pmatrix} A^e[:2u]\equiv 0 \mod x^{2u},$$
			and 
			$$(1+x^{2v})(1+x^v+x^{2v})^2=1+x^{6v}\equiv 0\mod x^{2u},$$
			we have
	\begin{align*}
		&\quad (1+x^{2v})\cdot B^e[:2u]A^e[:2u]\\
		&\equiv(1+x^{2v})\cdot\left((1+x^v+x^{2v})\cdot (B^e[:u]+x^v B^e[:u-v])+x^{4v}B^e[:2u-4v]\right)\\
			&\quad ((1+x^v+x^{2v})\cdot (A^e[:u]+x^v A^e[:u-v])+x^{4v}A^e[:2u-4v])\\
		&\quad(1+x^{2v})\cdot\left((1+x^v+x^{2v})\cdot (B^e[:u]+x^v B^e[:u-v])+x^{4v}B^e[:2u-4v]\right)\\
		&\quad  (x^u+x^{u+v}+x^{u+2v})\begin{pmatrix}1&1\\0&1\end{pmatrix}\\
		&\equiv (B^e[:u]+x^v B^e[:u-v])\cdot(A^e[:u]+x^v A^e[:u-v])\\
		&\quad x^u\cdot (B^e[:u]+x^v B^e[:u-v])\cdot
		\begin{pmatrix}1&1\\0&1\end{pmatrix} \mod x^{2u}.
	\end{align*}
	Thus we have proved that both sides of Eq. \eqref{eq:anp1} are congruent
	modulo $x^{2u}$, so that they must be equal as both have degree at
	most $2u$.
\end{proof}
	Theorem \ref{thm:pd1} can be derived from lemma \ref{lem:3} and the definition
	of $A^e$, using the same method as in the proof of theorem \ref{th:ab}.


\section{From equation to automaton}\label{sec:eq2kern}
In this section we give an description of the algorithm that we use to calculate
a $p$-automaton of an algebraic series $T(x)$ 
in $\mathbb{F}_{q}[[x]]$
from an annihilating polynomial of it, where $\mathbb{F}_q$ is a finite field
of characteristic $p$. The algorithm is based on the
proof of theorem 1 in  \cite{Christol1980KMFR}.

\medskip
\noindent
{\bf Step one: Normalization}\\
Input: an annihilating polynomial $P(x,y)\in \mathbb{F}_{q}(x)[y]$ of $T(x)$.\\
Output: an annihilating polynomial $Q(x,y)\in \mathbb{F}_{q}(x)[y]$ of $T(x)$ 
the form
$$y+\frac{a_1(x)}{b_1(x)}y^{p^1}+\frac{a_2(x)}{b_2(x)}y^{p^2}+\cdots+
\frac{a_n(x)}{b_n(x)}y^{p^n}.$$
Method:
use the relation $P(x,T(x))=0$ to express $T(x)^{p^j}$ as 
$\mathbb{F}_{p}(x)$-linear combination of $T(x)^k$, $k=0,1,\ldots, d-1$, 
where $d$ is the degree of $P(x,y)$ as a polynomial in $y$. In practice,
to find the expression of $T(x)^{p^j}$, we first calculate that of
$T(x)^{p^{j-1}}$, then raise it to the $p$-th power, and finally reduce
again using the relation $P(x,T(x))=0$.

We know that the family $T(x)^{p^j}$, 
$j=0,1,\ldots, d$ is necessarily linearly dependent. 
However, as it can be
costly to compute $T(x)^{p^j}$ when $j$ is large, in reality we stop once
the rank of the family $T(x)^{p^j}$, $k=0,1,\ldots, j_0$ is less than $j_0+1$.

\medskip
\noindent
{\bf Step two: From normalized equation to kernel}\\
Input: the relation
\begin{equation}\label{eq:rel}
	T(x)=\frac{a_1(x)}{b_1(x)}T(x)^{p^1}+\frac{a_2(x)}{b_2(x)}T(x)^{p^2}+\cdots+
\frac{a_n(x)}{b_n(x)}T(x)^{p^n}.
\end{equation}
Output: the $p$-kernel of $T(x)$.\\
Method: 
We let $\phi$ denote the Frobenius morphism and $\Lambda_j$ the Cartier operator
that maps $\sum a_l x^l$ to $\sum a_{pl+j} x^l$ for $j=0,1,\ldots, p-1$.
We recall that for a series $f(x)=\sum_{l\geq l_0} c_lx^l\in \mathbb{F}_q((x))$ 
and polynomials $a(x)$ and $b(x)$, 
$$\Lambda_j(a(x)f(x)^p)=\Lambda_j(a(x))\Lambda_0 (f(x)^p)$$
for $j=0,1,\ldots,p$ and
$$\Lambda_0(f(x)^p)=\Lambda_0 \sum_{l\geq l_0} c_l^p x^{p\cdot l}=\sum_{l\geq l_0} c_l^p x^l=\phi(f)(x).$$ 
Combining the above two identities and we get
\begin{equation}\label{eq:rule}
\Lambda_j\left(\frac{a(x)}{b(x)}f(x)^p\right)
=\Lambda_j\left(a(x)b(x)^{p-1} \frac{f(x)^p}{b(x)^p}\right)
=\Lambda_j(a(x)b(x)^{p-1})\frac{ \phi(f)(x)}{\phi(b)(x)}.
\end{equation}

When we apply repeatedly $\Lambda_j$, $j=0,1, \ldots, p-1$ 
to both sides of \eqref{eq:rel} using the above computation rule and rewrite
$\phi^k(T)(x)$ using relation \eqref{eq:rel}, 
we always get an expression of the form (this will be illustrated by 
example \ref{eg:ker} below)
\begin{equation}\label{eq:ker2}
	\frac{c_1(x)}{d_1(x)}\phi^k(T)(x)^{p^1}+\frac{c_2(x)}{d_2(x)}\phi^k(T)(x)^{p^2}+\cdots+
\frac{c_n(x)}{d_n(x)}\phi^k(T)(x)^{p^n},
\end{equation}
where $k=0,1,\ldots, \log q/\log p-1$, and $c_j(x)$ and $d_j(x)$ are polynomial 
of bounded degree for $j=1,2,\ldots, n$. 
To see the last point, note that for $d_j(x)$ is always a factor of 
$$\prod_{l=1}^{n} \phi^k(b_l(x))$$
for some $0\leq k < \log q/\log p$, 
and 
$$
\deg c_j\leq \deg d_j +
\max \{\deg a_l -\deg b_l\mid l=1,2,\ldots, n\}.
$$
The set of expression of the form \eqref{eq:ker2} is therefore finite and
the process must terminate. In the end we get a finite set that is the
$p$-kernel of $T(x)$.

In our program, the expression \eqref{eq:ker2} is encoded by the tuple
$$
\left(\left \{1: {c_1(x)}/{d_1(x)}, 2:{c_2(x)}/{d_2(x)}, \ldots, 
n: {c_n(x)}/{d_n(x)}\right\}, k\right).
$$
\begin{Remark}
Note that the reason we use powers of $p$ instead of powers of $q$ in 
\eqref{eq:rel} is that the latter usually needs much larger coefficients.
\end{Remark}
\medskip

\medskip
\begin{Example}\label{eg:ker}
Set $a=\bar{u}\in \mathbb{F}_4=\mathbb{F}_2[u]/<u^2+u+1>$.
Let $T(x)$ be the unique solution in $\mathbb{F}_4[[x]]$ of  
$$(x^2+ax)y^3+y+a+x=0.$$
We write the equation in the normalized form, which is really easy for this
example:
\begin{equation}\label{eq:eg}
	T(x)=\frac{1}{x+a}T(x)^2+x T(x)^4.
\end{equation}
	We use computation rule \eqref{eq:rule} to calculate $\Lambda_0 T(x)$ and $\Lambda_0\Lambda_0 T(x)$ to illustrate 
of this process:
\begin{align*}
	\Lambda_0 T(x)&=\Lambda_0(\frac{1}{x+a}T(x)^2)+\Lambda_0(x T(x)^4)\\
	&=\Lambda_0\left((x+a)\frac{T(x)^2}{(x+a)^2}\right)
	+\Lambda_0(x)\cdot \phi(T)(x)^2\\
	&=a \frac{\phi(T)(x)}{x+a+1}.
\end{align*}
	To calculate $\Lambda_0\Lambda_0 T(x)$, we need to first put the above
	expression into form \eqref{eq:ker2}. Applying $\phi$ to both sides of 
	\eqref{eq:eg} we get
\begin{equation*}
	\phi(	T)(x)=\frac{1}{x+a+1}\phi(T)(x)^2+x \phi(T)(x)^4.
\end{equation*}
Therefore 
\begin{align*}
	\Lambda_0 T(x)&=\frac{a}{(x+a+1)^2}\phi(T)(x)^2+\frac{ax}{(x+a+1)}
	\phi(T)(x)^4,
\end{align*}
and 
\begin{align*}
\Lambda_0	\Lambda_0 T(x)&=\Lambda_0\left(\frac{a}{(x+a+1)^2}\phi(T)(x)^2\right)
	+\Lambda_0\left(\frac{ax(x+a+1)}{(x+a+1)^2} \phi(T)(x)^4, \right)\\
	&=\frac{a}{x+a} \phi^2(T)(x)+ \frac{a}{x+a} \phi^2(T)(x)^2\\
	&=\frac{a}{x+a} T(x)+ \frac{ax}{x+a}T(x)^2.\\
\end{align*}
	In our program, the series $T(x)$, $\Lambda_0 T(x)$ and $\Lambda_0\Lambda_0
	T(x)$ are encoded by 
	$$(\{0:1\},\; 0),$$
	$$(\{0:a/(x + a + 1)\},\;1),$$
	and 
	$$(\{0: a/(x + a),\; 1: ax/(x + a)\},\; 0).$$

\end{Example}

\medskip
\noindent
{\bf Step three: Output function}
In step two, each element from the $p$-kernel of $T(x)$ is expressed as an 
$\mathbb{F}_{q}(x)$-linear combination of powers of $\phi^k(T)(x)$, for 
some $0\leq k< \log q/\log p$.  The output
function maps the corresponding state to the constant term of the series.
To calculate it, we simply plug in $T(x) \mod x^{D+1}$, where $D$ is the maximum
of $0$ and the minus of the orders of the coefficients of the 
linear combination.

\newpage
	\section{Annex}

\subsection{Data for Section \ref{section:tmncf}}\label{data:tmncf}
All $16$ polynomials are of the form  
$$p_0(x)+p_3(x)y^3+p_6(x)y^6+p_9(x)y^9+p_{12}(x)y^{12}.$$
The coefficients $p_j(x)$ and the $8$ initial terms to determine the solutions uniquely
are given below.

For $\phi(M^e,0,0)$:
\begin{align*}
p_{0}(x)&=x^{66} +x^{64} +x^{62} +x^{60} +x^{58} +x^{56} +x^{52} +x^{50} +x^{36} +x^{32} +x^{30}\\
& \quad +x^{20} +x^{16} +x^{14} +x^{12} ,\\
p_{3}(x)&=x^{62} +x^{60} +x^{58} +x^{56} +x^{52} +x^{50} +x^{48} +x^{44} +x^{42} +x^{38} +x^{36}\\
& \quad +x^{32} +x^{28} +x^{22} +x^{20} +x^{18} +x^{14} +x^{12} ,\\
p_{6}(x)&=x^{56} +x^{44} +x^{40} +x^{38} +x^{36} +x^{32} +x^{30} +x^{26} +x^{20} +x^{18} +x^{16}\\
& \quad +x^{14} +x^2 +1 ,\\
p_{9}(x)&=x^{64} +x^{62} +x^{58} +x^{56} +x^{54} +x^{52} +x^{48} +x^{42} +x^{32} +x^{30} +x^{26}\\
& \quad +x^{20} +x^{18} +x^{14} +x^{12} +x^{10} +x^8 +x^2 ,\\
p_{12}(x)&=x^{64} +x^{56} +x^{40} +x^{32} +x^{16} +x^8 +1 ,
\end{align*}
and the initial terms are [1, 0, 0, 0, 0, 0, 0, 0].

 For $\phi(M^e,0,1)$:
\begin{align*}
p_{0}(x)&=x^{63} +x^{62} +x^{60} +x^{59} +x^{56} +x^{55} +x^{53} +x^{51} +x^{50} +x^{49} +x^{48}\\
& \quad +x^{47} +x^{45} +x^{44} +x^{42} +x^{41} +x^{39} +x^{37} +x^{36} +x^{35} +x^{34} +x^{32}\\
& \quad +x^{28} +x^{27} +x^{25} +x^{24} +x^{23} +x^{21} +x^{20} +x^{18} +x^{15} ,\\
p_{3}(x)&=x^{61} +x^{59} +x^{51} +x^{49} +x^{43} +x^{41} +x^{37} +x^{33} +x^{21} +x^{19} +x^{13} +x^9 ,\\
p_{6}(x)&=x^{60} +x^{59} +x^{57} +x^{55} +x^{54} +x^{52} +x^{51} +x^{49} +x^{48} +x^{46} +x^{45}\\
& \quad +x^{43} +x^{42} +x^{32} +x^{31} +x^{29} +x^{28} +x^{26} +x^{25} +x^{24} +x^{22} +x^{21}\\
& \quad +x^{19} +x^{18} +x^{16} +x^{15} +x^{13} +x^{12} +x^{10} +x^9 +x^7 +x^6 ,\\
p_{9}(x)&=x^{59} +x^{53} +x^{51} +x^{49} +x^{47} +x^{45} +x^{43} +x^{41} +x^{37} +x^{29} +x^{27}\\
& \quad +x^{25} +x^{21} +x^{15} +x^{13} +x^7 +x^5 +x^3 ,\\
p_{12}(x)&=x^{58} +x^{57} +x^{56} +x^{54} +x^{53} +x^{52} +x^{50} +x^{49} +x^{48} +x^{42} +x^{41}\\
& \quad +x^{40} +x^{38} +x^{37} +x^{36} +x^{34} +x^{33} +x^{32} +x^{10} +x^9 +x^8 +x^6 +x^5\\
& \quad +x^4 +x^2 +x +1 ,
\end{align*}
and the initial terms are [0, 1, 0, 0, 1, 1, 0, 0].

 For $\phi(M^e,1,0)$:
\begin{align*}
p_{0}(x)&=x^{63} +x^{62} +x^{60} +x^{59} +x^{56} +x^{55} +x^{53} +x^{51} +x^{50} +x^{49} +x^{48}\\
& \quad +x^{47} +x^{45} +x^{44} +x^{42} +x^{41} +x^{39} +x^{37} +x^{36} +x^{35} +x^{34} +x^{32}\\
& \quad +x^{28} +x^{27} +x^{25} +x^{24} +x^{23} +x^{21} +x^{20} +x^{18} +x^{15} ,\\
p_{3}(x)&=x^{61} +x^{59} +x^{51} +x^{49} +x^{43} +x^{41} +x^{37} +x^{33} +x^{21} +x^{19} +x^{13} +x^9 ,\\
p_{6}(x)&=x^{60} +x^{59} +x^{57} +x^{55} +x^{54} +x^{52} +x^{51} +x^{49} +x^{48} +x^{46} +x^{45}\\
& \quad +x^{43} +x^{42} +x^{32} +x^{31} +x^{29} +x^{28} +x^{26} +x^{25} +x^{24} +x^{22} +x^{21}\\
& \quad +x^{19} +x^{18} +x^{16} +x^{15} +x^{13} +x^{12} +x^{10} +x^9 +x^7 +x^6 ,\\
p_{9}(x)&=x^{59} +x^{53} +x^{51} +x^{49} +x^{47} +x^{45} +x^{43} +x^{41} +x^{37} +x^{29} +x^{27}\\
& \quad +x^{25} +x^{21} +x^{15} +x^{13} +x^7 +x^5 +x^3 ,\\
p_{12}(x)&=x^{58} +x^{57} +x^{56} +x^{54} +x^{53} +x^{52} +x^{50} +x^{49} +x^{48} +x^{42} +x^{41}\\
& \quad +x^{40} +x^{38} +x^{37} +x^{36} +x^{34} +x^{33} +x^{32} +x^{10} +x^9 +x^8 +x^6 +x^5\\
& \quad +x^4 +x^2 +x +1 ,
\end{align*}
and the initial terms are [0, 1, 0, 0, 1, 1, 0, 0].

 For $\phi(M^e,1,1)$:
\begin{align*}
p_{0}(x)&=x^{60} +x^{56} +x^{54} +x^{50} +x^{48} +x^{44} +x^{38} +x^{34} +x^{30} +x^{28} +x^{26}\\
& \quad +x^{20} +x^{18} ,\\
p_{3}(x)&=x^{54} +x^{50} +x^{44} +x^{38} +x^{28} +x^{20} +x^{18} +x^{12} +x^{10} +x^6 ,\\
p_{6}(x)&=x^{54} +x^{52} +x^{50} +x^{38} +x^{36} +x^{24} +x^{18} +x^{16} +x^{10} +x^6 +x^4 +1 ,\\
p_{9}(x)&=x^{52} +x^{44} +x^{42} +x^{40} +x^{30} +x^{24} +x^{22} +x^{20} +x^{16} +x^{10} +x^8 +x^6\\
& \quad +x^4 +x^2 ,\\
p_{12}(x)&=x^{52} +x^{50} +x^{48} +x^{36} +x^{34} +x^{32} +x^4 +x^2 +1 ,
\end{align*}
and the initial terms are [0, 0, 1, 0, 0, 0, 0, 0].

 For $\phi(W^e,0,0)$:
\begin{align*}
p_{0}(x)&=x^{60} +x^{50} +x^{44} +x^{40} +x^{30} +x^{28} +x^{24} +x^{18} +x^{12} ,\\
p_{3}(x)&=x^{54} +x^{44} +x^{34} +x^{28} +x^{26} +x^{22} +x^{20} +x^{14} +x^{12} +x^{10} ,\\
p_{6}(x)&=x^{54} +x^{52} +x^{48} +x^{46} +x^{42} +x^{38} +x^{34} +x^{32} +x^{30} +x^{22} +x^{20}\\
& \quad +x^{18} +x^{14} +x^{12} +x^6 +1 ,\\
p_{9}(x)&=x^{52} +x^{44} +x^{42} +x^{40} +x^{30} +x^{24} +x^{22} +x^{20} +x^{16} +x^{10} +x^8 +x^6\\
& \quad +x^4 +x^2 ,\\
p_{12}(x)&=x^{52} +x^{50} +x^{48} +x^{36} +x^{34} +x^{32} +x^4 +x^2 +1 ,
\end{align*}
and the initial terms are [1, 0, 1, 0, 1, 0, 0, 0].

 For $\phi(W^e,0,1)$:
\begin{align*}
p_{0}(x)&=x^{63} +x^{60} +x^{59} +x^{56} +x^{52} +x^{51} +x^{50} +x^{47} +x^{43} +x^{39} +x^{36}\\
& \quad +x^{35} +x^{28} +x^{25} +x^{22} +x^{20} +x^{18} +x^{16} +x^{15} ,\\
p_{3}(x)&=x^{61} +x^{59} +x^{51} +x^{49} +x^{43} +x^{41} +x^{37} +x^{33} +x^{21} +x^{19} +x^{13} +x^9 ,\\
p_{6}(x)&=x^{60} +x^{59} +x^{57} +x^{55} +x^{54} +x^{52} +x^{51} +x^{49} +x^{48} +x^{46} +x^{45}\\
& \quad +x^{43} +x^{42} +x^{32} +x^{31} +x^{29} +x^{28} +x^{26} +x^{25} +x^{24} +x^{22} +x^{21}\\
& \quad +x^{19} +x^{18} +x^{16} +x^{15} +x^{13} +x^{12} +x^{10} +x^9 +x^7 +x^6 ,\\
p_{9}(x)&=x^{59} +x^{53} +x^{51} +x^{49} +x^{47} +x^{45} +x^{43} +x^{41} +x^{37} +x^{29} +x^{27}\\
& \quad +x^{25} +x^{21} +x^{15} +x^{13} +x^7 +x^5 +x^3 ,\\
p_{12}(x)&=x^{58} +x^{57} +x^{56} +x^{54} +x^{53} +x^{52} +x^{50} +x^{49} +x^{48} +x^{42} +x^{41}\\
& \quad +x^{40} +x^{38} +x^{37} +x^{36} +x^{34} +x^{33} +x^{32} +x^{10} +x^9 +x^8 +x^6 +x^5\\
& \quad +x^4 +x^2 +x +1 ,
\end{align*}
and the initial terms are [0, 0, 1, 1, 1, 1, 0, 1].

 For $\phi(W^e,1,0)$:
\begin{align*}
p_{0}(x)&=x^{63} +x^{60} +x^{59} +x^{56} +x^{52} +x^{51} +x^{50} +x^{47} +x^{43} +x^{39} +x^{36}\\
& \quad +x^{35} +x^{28} +x^{25} +x^{22} +x^{20} +x^{18} +x^{16} +x^{15} ,\\
p_{3}(x)&=x^{61} +x^{59} +x^{51} +x^{49} +x^{43} +x^{41} +x^{37} +x^{33} +x^{21} +x^{19} +x^{13} +x^9 ,\\
p_{6}(x)&=x^{60} +x^{59} +x^{57} +x^{55} +x^{54} +x^{52} +x^{51} +x^{49} +x^{48} +x^{46} +x^{45}\\
& \quad +x^{43} +x^{42} +x^{32} +x^{31} +x^{29} +x^{28} +x^{26} +x^{25} +x^{24} +x^{22} +x^{21}\\
& \quad +x^{19} +x^{18} +x^{16} +x^{15} +x^{13} +x^{12} +x^{10} +x^9 +x^7 +x^6 ,\\
p_{9}(x)&=x^{59} +x^{53} +x^{51} +x^{49} +x^{47} +x^{45} +x^{43} +x^{41} +x^{37} +x^{29} +x^{27}\\
& \quad +x^{25} +x^{21} +x^{15} +x^{13} +x^7 +x^5 +x^3 ,\\
p_{12}(x)&=x^{58} +x^{57} +x^{56} +x^{54} +x^{53} +x^{52} +x^{50} +x^{49} +x^{48} +x^{42} +x^{41}\\
& \quad +x^{40} +x^{38} +x^{37} +x^{36} +x^{34} +x^{33} +x^{32} +x^{10} +x^9 +x^8 +x^6 +x^5\\
& \quad +x^4 +x^2 +x +1 ,
\end{align*}
and the initial terms are [0, 0, 1, 1, 1, 1, 0, 1].

 For $\phi(W^e,1,1)$:
\begin{align*}
p_{0}(x)&=x^{66} +x^{60} +x^{58} +x^{56} +x^{52} +x^{48} +x^{42} +x^{40} +x^{36} +x^{32} +x^{30}\\
& \quad +x^{28} +x^{26} +x^{20} +x^{18} ,\\
p_{3}(x)&=x^{58} +x^{50} +x^{46} +x^{44} +x^{42} +x^{38} +x^{36} +x^{32} +x^{30} +x^{24} +x^{20}\\
& \quad +x^{18} +x^8 +x^6 ,\\
p_{6}(x)&=x^{62} +x^{56} +x^{52} +x^{46} +x^{40} +x^{38} +x^{36} +x^{32} +x^{26} +x^{22} +x^{18}\\
& \quad +x^{16} +x^6 +x^4 +x^2 +1 ,\\
p_{9}(x)&=x^{64} +x^{62} +x^{58} +x^{56} +x^{54} +x^{52} +x^{48} +x^{42} +x^{32} +x^{30} +x^{26}\\
& \quad +x^{20} +x^{18} +x^{14} +x^{12} +x^{10} +x^8 +x^2 ,\\
p_{12}(x)&=x^{64} +x^{56} +x^{40} +x^{32} +x^{16} +x^8 +1 ,
\end{align*}
and the initial terms are [0, 0, 0, 0, 1, 0, 0, 0].

 For $\phi(M^o,0,0)$:
\begin{align*}
p_{0}(x)&=x^{63} +x^{62} +x^{61} +x^{60} +x^{59} +x^{57} +x^{55} +x^{53} +x^{51} +x^{50} +x^{49}\\
& \quad +x^{47} +x^{45} +x^{44} +x^{42} +x^{40} +x^{39} +x^{37} +x^{36} +x^{31} +x^{30} +x^{28}\\
& \quad +x^{27} +x^{26} +x^{24} +x^{23} +x^{21} +x^{18} +x^{17} +x^{15} +x^{14} +x^{13} +x^{12} ,\\
p_{3}(x)&=x^{62} +x^{59} +x^{58} +x^{56} +x^{55} +x^{53} +x^{52} +x^{51} +x^{46} +x^{45} +x^{41}\\
& \quad +x^{39} +x^{38} +x^{37} +x^{36} +x^{34} +x^{32} +x^{31} +x^{30} +x^{28} +x^{27} +x^{25}\\
& \quad +x^{23} +x^{20} +x^{18} +x^{17} +x^{16} +x^{13} +x^{11} +x^{10} +x^9 +x^8 ,\\
p_{6}(x)&=x^{62} +x^{59} +x^{58} +x^{57} +x^{54} +x^{52} +x^{51} +x^{48} +x^{44} +x^{42} +x^{41}\\
& \quad +x^{37} +x^{36} +x^{35} +x^{33} +x^{31} +x^{30} +x^{29} +x^{26} +x^{24} +x^{23} +x^{18}\\
& \quad +x^{17} +x^{14} +x^{13} +x^{12} +x^{10} +x^8 +x^7 +x^2 +x +1 ,\\
p_{9}(x)&=x^{62} +x^{61} +x^{60} +x^{54} +x^{53} +x^{52} +x^{46} +x^{45} +x^{44} +x^{38} +x^{37}\\
& \quad +x^{36} +x^{14} +x^{13} +x^{12} +x^6 +x^5 +x^4 ,\\
p_{12}(x)&=x^{62} +x^{61} +x^{60} +x^{58} +x^{57} +x^{56} +x^{50} +x^{49} +x^{48} +x^{46} +x^{45}\\
& \quad +x^{44} +x^{42} +x^{41} +x^{40} +x^{34} +x^{33} +x^{32} +x^{14} +x^{13} +x^{12} +x^{10}\\
& \quad +x^9 +x^8 +x^2 +x +1 ,
\end{align*}
and the initial terms are [1, 0, 1, 1, 0, 1, 0, 0].

 For $\phi(M^o,0,1)$:
\begin{align*}
p_{0}(x)&=x^{64} +x^{61} +x^{59} +x^{55} +x^{54} +x^{52} +x^{51} +x^{50} +x^{48} +x^{47} +x^{46}\\
& \quad +x^{45} +x^{42} +x^{41} +x^{40} +x^{38} +x^{36} +x^{34} +x^{33} +x^{32} +x^{31} +x^{30}\\
& \quad +x^{29} +x^{28} +x^{26} +x^{23} +x^{20} +x^{19} +x^{17} +x^{16} +x^{15} ,\\
p_{3}(x)&=x^{57} +x^{56} +x^{55} +x^{51} +x^{50} +x^{48} +x^{47} +x^{43} +x^{42} +x^{40} +x^{39}\\
& \quad +x^{35} +x^{34} +x^{33} +x^{17} +x^{16} +x^{15} +x^{11} +x^{10} +x^9 ,\\
p_{6}(x)&=x^{58} +x^{56} +x^{54} +x^{50} +x^{48} +x^{44} +x^{42} +x^{30} +x^{28} +x^{24} +x^{20}\\
& \quad +x^{18} +x^{14} +x^{12} +x^8 +x^6 ,\\
p_{9}(x)&=x^{59} +x^{58} +x^{56} +x^{54} +x^{52} +x^{51} +x^{43} +x^{42} +x^{40} +x^{38} +x^{36}\\
& \quad +x^{35} +x^{11} +x^{10} +x^8 +x^6 +x^4 +x^3 ,\\
p_{12}(x)&=x^{60} +x^{56} +x^{48} +x^{44} +x^{40} +x^{32} +x^{12} +x^8 +1 ,
\end{align*}
and the initial terms are [0, 1, 0, 1, 0, 1, 0, 0].

 For $\phi(M^o,1,0)$:
\begin{align*}
p_{0}(x)&=x^{66} +x^{65} +x^{64} +x^{62} +x^{60} +x^{59} +x^{58} +x^{54} +x^{51} +x^{46} +x^{45}\\
& \quad +x^{43} +x^{42} +x^{41} +x^{40} +x^{35} +x^{33} +x^{31} +x^{29} +x^{28} +x^{27} +x^{23}\\
& \quad +x^{22} +x^{21} +x^{20} +x^{18} +x^{15} ,\\
p_{3}(x)&=x^{65} +x^{64} +x^{63} +x^{61} +x^{60} +x^{58} +x^{57} +x^{55} +x^{54} +x^{52} +x^{51}\\
& \quad +x^{47} +x^{46} +x^{44} +x^{43} +x^{41} +x^{40} +x^{38} +x^{37} +x^{35} +x^{34} +x^{33}\\
& \quad +x^{25} +x^{24} +x^{23} +x^{21} +x^{20} +x^{18} +x^{16} +x^{14} +x^{13} +x^{11} +x^{10}\\
& \quad +x^9 ,\\
p_{6}(x)&=x^{66} +x^{64} +x^{60} +x^{58} +x^{46} +x^{44} +x^{42} +x^{38} +x^{36} +x^{34} +x^{30}\\
& \quad +x^{28} +x^{26} +x^{10} +x^8 +x^6 ,\\
p_{9}(x)&=x^{67} +x^{66} +x^{64} +x^{63} +x^{55} +x^{54} +x^{52} +x^{50} +x^{48} +x^{47} +x^{39}\\
& \quad +x^{38} +x^{36} +x^{35} +x^{19} +x^{18} +x^{16} +x^{15} +x^7 +x^6 +x^4 +x^3 ,\\
p_{12}(x)&=x^{68} +x^{48} +x^{36} +x^{32} +x^{20} +x^4 +1 ,
\end{align*}
and the initial terms are [0, 0, 1, 1, 1, 0, 0, 1].

 For $\phi(M^o,1,1)$:
\begin{align*}
p_{0}(x)&=x^{63} +x^{61} +x^{59} +x^{58} +x^{57} +x^{56} +x^{55} +x^{53} +x^{49} +x^{48} +x^{45}\\
& \quad +x^{44} +x^{43} +x^{42} +x^{35} +x^{31} +x^{30} +x^{29} +x^{28} +x^{26} +x^{25} +x^{20}\\
& \quad +x^{18} ,\\
p_{3}(x)&=x^{62} +x^{60} +x^{56} +x^{55} +x^{53} +x^{52} +x^{51} +x^{50} +x^{49} +x^{48} +x^{46}\\
& \quad +x^{45} +x^{42} +x^{40} +x^{39} +x^{38} +x^{37} +x^{35} +x^{34} +x^{33} +x^{31} +x^{30}\\
& \quad +x^{28} +x^{27} +x^{25} +x^{23} +x^{19} +x^{17} +x^{16} +x^{13} +x^{12} +x^{10} ,\\
p_{6}(x)&=x^{62} +x^{60} +x^{57} +x^{56} +x^{55} +x^{49} +x^{46} +x^{45} +x^{44} +x^{43} +x^{41}\\
& \quad +x^{38} +x^{35} +x^{34} +x^{30} +x^{28} +x^{25} +x^{23} +x^{20} +x^{19} +x^{18} +x^{15}\\
& \quad +x^{13} +x^{12} +x^{10} +x^5 +x^4 +x^2 +x +1 ,\\
p_{9}(x)&=x^{62} +x^{61} +x^{60} +x^{54} +x^{53} +x^{52} +x^{46} +x^{45} +x^{44} +x^{38} +x^{37}\\
& \quad +x^{36} +x^{14} +x^{13} +x^{12} +x^6 +x^5 +x^4 ,\\
p_{12}(x)&=x^{62} +x^{61} +x^{60} +x^{58} +x^{57} +x^{56} +x^{50} +x^{49} +x^{48} +x^{46} +x^{45}\\
& \quad +x^{44} +x^{42} +x^{41} +x^{40} +x^{34} +x^{33} +x^{32} +x^{14} +x^{13} +x^{12} +x^{10}\\
& \quad +x^9 +x^8 +x^2 +x +1 ,
\end{align*}
and the initial terms are [0, 0, 0, 1, 1, 1, 1, 0].

 For $\phi(W^o,0,0)$:
\begin{align*}
p_{0}(x)&=x^{63} +x^{62} +x^{61} +x^{60} +x^{59} +x^{57} +x^{55} +x^{53} +x^{51} +x^{50} +x^{49}\\
& \quad +x^{47} +x^{45} +x^{44} +x^{42} +x^{40} +x^{39} +x^{37} +x^{36} +x^{31} +x^{30} +x^{28}\\
& \quad +x^{27} +x^{26} +x^{24} +x^{23} +x^{21} +x^{18} +x^{17} +x^{15} +x^{14} +x^{13} +x^{12} ,\\
p_{3}(x)&=x^{62} +x^{59} +x^{58} +x^{56} +x^{55} +x^{53} +x^{52} +x^{51} +x^{46} +x^{45} +x^{41}\\
& \quad +x^{39} +x^{38} +x^{37} +x^{36} +x^{34} +x^{32} +x^{31} +x^{30} +x^{28} +x^{27} +x^{25}\\
& \quad +x^{23} +x^{20} +x^{18} +x^{17} +x^{16} +x^{13} +x^{11} +x^{10} +x^9 +x^8 ,\\
p_{6}(x)&=x^{62} +x^{59} +x^{58} +x^{57} +x^{54} +x^{52} +x^{51} +x^{48} +x^{44} +x^{42} +x^{41}\\
& \quad +x^{37} +x^{36} +x^{35} +x^{33} +x^{31} +x^{30} +x^{29} +x^{26} +x^{24} +x^{23} +x^{18}\\
& \quad +x^{17} +x^{14} +x^{13} +x^{12} +x^{10} +x^8 +x^7 +x^2 +x +1 ,\\
p_{9}(x)&=x^{62} +x^{61} +x^{60} +x^{54} +x^{53} +x^{52} +x^{46} +x^{45} +x^{44} +x^{38} +x^{37}\\
& \quad +x^{36} +x^{14} +x^{13} +x^{12} +x^6 +x^5 +x^4 ,\\
p_{12}(x)&=x^{62} +x^{61} +x^{60} +x^{58} +x^{57} +x^{56} +x^{50} +x^{49} +x^{48} +x^{46} +x^{45}\\
& \quad +x^{44} +x^{42} +x^{41} +x^{40} +x^{34} +x^{33} +x^{32} +x^{14} +x^{13} +x^{12} +x^{10}\\
& \quad +x^9 +x^8 +x^2 +x +1 ,
\end{align*}
and the initial terms are [1, 0, 1, 1, 0, 1, 0, 0].

 For $\phi(W^o,0,1)$:
\begin{align*}
p_{0}(x)&=x^{66} +x^{65} +x^{64} +x^{62} +x^{60} +x^{59} +x^{58} +x^{54} +x^{51} +x^{46} +x^{45}\\
& \quad +x^{43} +x^{42} +x^{41} +x^{40} +x^{35} +x^{33} +x^{31} +x^{29} +x^{28} +x^{27} +x^{23}\\
& \quad +x^{22} +x^{21} +x^{20} +x^{18} +x^{15} ,\\
p_{3}(x)&=x^{65} +x^{64} +x^{63} +x^{61} +x^{60} +x^{58} +x^{57} +x^{55} +x^{54} +x^{52} +x^{51}\\
& \quad +x^{47} +x^{46} +x^{44} +x^{43} +x^{41} +x^{40} +x^{38} +x^{37} +x^{35} +x^{34} +x^{33}\\
& \quad +x^{25} +x^{24} +x^{23} +x^{21} +x^{20} +x^{18} +x^{16} +x^{14} +x^{13} +x^{11} +x^{10}\\
& \quad +x^9 ,\\
p_{6}(x)&=x^{66} +x^{64} +x^{60} +x^{58} +x^{46} +x^{44} +x^{42} +x^{38} +x^{36} +x^{34} +x^{30}\\
& \quad +x^{28} +x^{26} +x^{10} +x^8 +x^6 ,\\
p_{9}(x)&=x^{67} +x^{66} +x^{64} +x^{63} +x^{55} +x^{54} +x^{52} +x^{50} +x^{48} +x^{47} +x^{39}\\
& \quad +x^{38} +x^{36} +x^{35} +x^{19} +x^{18} +x^{16} +x^{15} +x^7 +x^6 +x^4 +x^3 ,\\
p_{12}(x)&=x^{68} +x^{48} +x^{36} +x^{32} +x^{20} +x^4 +1 ,
\end{align*}
and the initial terms are [0, 0, 1, 1, 1, 0, 0, 1].

 For $\phi(W^o,1,0)$:
\begin{align*}
p_{0}(x)&=x^{64} +x^{61} +x^{59} +x^{55} +x^{54} +x^{52} +x^{51} +x^{50} +x^{48} +x^{47} +x^{46}\\
& \quad +x^{45} +x^{42} +x^{41} +x^{40} +x^{38} +x^{36} +x^{34} +x^{33} +x^{32} +x^{31} +x^{30}\\
& \quad +x^{29} +x^{28} +x^{26} +x^{23} +x^{20} +x^{19} +x^{17} +x^{16} +x^{15} ,\\
p_{3}(x)&=x^{57} +x^{56} +x^{55} +x^{51} +x^{50} +x^{48} +x^{47} +x^{43} +x^{42} +x^{40} +x^{39}\\
& \quad +x^{35} +x^{34} +x^{33} +x^{17} +x^{16} +x^{15} +x^{11} +x^{10} +x^9 ,\\
p_{6}(x)&=x^{58} +x^{56} +x^{54} +x^{50} +x^{48} +x^{44} +x^{42} +x^{30} +x^{28} +x^{24} +x^{20}\\
& \quad +x^{18} +x^{14} +x^{12} +x^8 +x^6 ,\\
p_{9}(x)&=x^{59} +x^{58} +x^{56} +x^{54} +x^{52} +x^{51} +x^{43} +x^{42} +x^{40} +x^{38} +x^{36}\\
& \quad +x^{35} +x^{11} +x^{10} +x^8 +x^6 +x^4 +x^3 ,\\
p_{12}(x)&=x^{60} +x^{56} +x^{48} +x^{44} +x^{40} +x^{32} +x^{12} +x^8 +1 ,
\end{align*}
and the initial terms are [0, 1, 0, 1, 0, 1, 0, 0].

 For $\phi(W^o,1,1)$:
\begin{align*}
p_{0}(x)&=x^{63} +x^{61} +x^{59} +x^{58} +x^{57} +x^{56} +x^{55} +x^{53} +x^{49} +x^{48} +x^{45}\\
& \quad +x^{44} +x^{43} +x^{42} +x^{35} +x^{31} +x^{30} +x^{29} +x^{28} +x^{26} +x^{25} +x^{20}\\
& \quad +x^{18} ,\\
p_{3}(x)&=x^{62} +x^{60} +x^{56} +x^{55} +x^{53} +x^{52} +x^{51} +x^{50} +x^{49} +x^{48} +x^{46}\\
& \quad +x^{45} +x^{42} +x^{40} +x^{39} +x^{38} +x^{37} +x^{35} +x^{34} +x^{33} +x^{31} +x^{30}\\
& \quad +x^{28} +x^{27} +x^{25} +x^{23} +x^{19} +x^{17} +x^{16} +x^{13} +x^{12} +x^{10} ,\\
p_{6}(x)&=x^{62} +x^{60} +x^{57} +x^{56} +x^{55} +x^{49} +x^{46} +x^{45} +x^{44} +x^{43} +x^{41}\\
& \quad +x^{38} +x^{35} +x^{34} +x^{30} +x^{28} +x^{25} +x^{23} +x^{20} +x^{19} +x^{18} +x^{15}\\
& \quad +x^{13} +x^{12} +x^{10} +x^5 +x^4 +x^2 +x +1 ,\\
p_{9}(x)&=x^{62} +x^{61} +x^{60} +x^{54} +x^{53} +x^{52} +x^{46} +x^{45} +x^{44} +x^{38} +x^{37}\\
& \quad +x^{36} +x^{14} +x^{13} +x^{12} +x^6 +x^5 +x^4 ,\\
p_{12}(x)&=x^{62} +x^{61} +x^{60} +x^{58} +x^{57} +x^{56} +x^{50} +x^{49} +x^{48} +x^{46} +x^{45}\\
& \quad +x^{44} +x^{42} +x^{41} +x^{40} +x^{34} +x^{33} +x^{32} +x^{14} +x^{13} +x^{12} +x^{10}\\
& \quad +x^9 +x^8 +x^2 +x +1 ,
\end{align*}
and the initial terms are [0, 0, 0, 1, 1, 1, 1, 0].

Below is the transition function and output function of an $2$-automaton that 
generates $T=M^o_{0,0}$:

Transition function $(n,j)\mapsto \delta(n,j)$ ($\Lambda(n):=[\delta(n,0),\delta(n,1)]$):
{\small

{
\renewcommand{\arraystretch}{1.2}

\begin{longtable}{| c c  |  c c  |  c c  |  c c  |  c c  |}
\hline
$n$ & $\Lambda(n)$ & $n$ & $\Lambda(n)$ & $n$ & $\Lambda(n)$ & $n$ & $\Lambda(n)$ & $n$ & $\Lambda(n)$ \\
\hline
0 & [1, 2]& 25 & [14, 42]& 50 & [71, 44]& 75 & [99, 51]& 100 & [10, 39]\\
1 & [3, 4]& 26 & [26, 26]& 51 & [78, 29]& 76 & [100, 97]& 101 & [28, 9]\\
2 & [5, 6]& 27 & [43, 44]& 52 & [32, 14]& 77 & [16, 101]& 102 & [84, 36]\\
3 & [7, 8]& 28 & [45, 46]& 53 & [44, 79]& 78 & [102, 101]& 103 & [79, 41]\\
4 & [9, 10]& 29 & [47, 48]& 54 & [80, 81]& 79 & [103, 104]& 104 & [115, 40]\\
5 & [11, 12]& 30 & [49, 50]& 55 & [82, 36]& 80 & [54, 95]& 105 & [81, 93]\\
6 & [13, 14]& 31 & [51, 52]& 56 & [83, 84]& 81 & [64, 105]& 106 & [87, 81]\\
7 & [15, 16]& 32 & [53, 54]& 57 & [30, 71]& 82 & [48, 60]& 107 & [116, 42]\\
8 & [17, 18]& 33 & [55, 20]& 58 & [85, 79]& 83 & [61, 49]& 108 & [117, 35]\\
9 & [19, 20]& 34 & [56, 51]& 59 & [86, 87]& 84 & [97, 23]& 109 & [118, 41]\\
10 & [20, 21]& 35 & [57, 49]& 60 & [88, 37]& 85 & [23, 19]& 110 & [52, 56]\\
11 & [22, 23]& 36 & [58, 59]& 61 & [73, 31]& 86 & [106, 64]& 111 & [95, 59]\\
12 & [24, 19]& 37 & [60, 61]& 62 & [89, 75]& 87 & [59, 68]& 112 & [101, 104]\\
13 & [25, 26]& 38 & [62, 56]& 63 & [40, 2]& 88 & [107, 105]& 113 & [69, 18]\\
14 & [27, 4]& 39 & [63, 64]& 64 & [90, 26]& 89 & [108, 98]& 114 & [18, 27]\\
15 & [12, 28]& 40 & [21, 16]& 65 & [1, 75]& 90 & [105, 54]& 115 & [119, 106]\\
16 & [29, 30]& 41 & [65, 47]& 66 & [42, 12]& 91 & [6, 47]& 116 & [104, 98]\\
17 & [31, 32]& 42 & [66, 60]& 67 & [91, 92]& 92 & [109, 27]& 117 & [120, 61]\\
18 & [33, 9]& 43 & [67, 68]& 68 & [93, 87]& 93 & [68, 106]& 118 & [121, 21]\\
19 & [34, 30]& 44 & [50, 69]& 69 & [77, 84]& 94 & [110, 2]& 119 & [37, 92]\\
20 & [35, 29]& 45 & [70, 71]& 70 & [46, 6]& 95 & [111, 93]& 120 & [122, 28]\\
21 & [36, 37]& 46 & [72, 73]& 71 & [94, 95]& 96 & [112, 39]& 121 & [92, 89]\\
22 & [2, 38]& 47 & [74, 10]& 72 & [96, 48]& 97 & [8, 77]& 122 & [123, 69]\\
23 & [39, 40]& 48 & [75, 73]& 73 & [4, 97]& 98 & [113, 31]& 123 & [38, 8]\\
24 & [41, 35]& 49 & [76, 77]& 74 & [98, 46]& 99 & [114, 32]&  & \\
\hline
\end{longtable}

}

}

 Output function $n\mapsto \tau(n)$:
{\small

{
\renewcommand{\arraystretch}{1.2}

\begin{longtable}{| c c  |  c c  |  c c  |  c c  |  c c  |  c c  |  c c  |}
\hline
$n$ & $\tau(n)$ & $n$ & $\tau(n)$ & $n$ & $\tau(n)$ & $n$ & $\tau(n)$ & $n$ & $\tau(n)$ & $n$ & $\tau(n)$ & $n$ & $\tau(n)$ \\
\hline
0 & 0& 18 & 0& 36 & 1& 54 & 1& 72 & 1& 90 & 1& 108 & 1\\
1 & 0& 19 & 1& 37 & 0& 55 & 0& 73 & 1& 91 & 1& 109 & 1\\
2 & 0& 20 & 0& 38 & 1& 56 & 1& 74 & 0& 92 & 1& 110 & 0\\
3 & 0& 21 & 1& 39 & 1& 57 & 0& 75 & 0& 93 & 1& 111 & 0\\
4 & 1& 22 & 0& 40 & 1& 58 & 1& 76 & 0& 94 & 0& 112 & 1\\
5 & 0& 23 & 1& 41 & 0& 59 & 0& 77 & 0& 95 & 0& 113 & 0\\
6 & 1& 24 & 0& 42 & 1& 60 & 0& 78 & 1& 96 & 1& 114 & 0\\
7 & 0& 25 & 1& 43 & 1& 61 & 1& 79 & 0& 97 & 1& 115 & 0\\
8 & 1& 26 & 0& 44 & 0& 62 & 1& 80 & 1& 98 & 0& 116 & 0\\
9 & 1& 27 & 1& 45 & 1& 63 & 1& 81 & 1& 99 & 0& 117 & 1\\
10 & 0& 28 & 1& 46 & 1& 64 & 1& 82 & 0& 100 & 0& 118 & 1\\
11 & 0& 29 & 0& 47 & 0& 65 & 0& 83 & 1& 101 & 1& 119 & 0\\
12 & 0& 30 & 0& 48 & 0& 66 & 1& 84 & 1& 102 & 1& 120 & 1\\
13 & 1& 31 & 1& 49 & 0& 67 & 1& 85 & 1& 103 & 0& 121 & 1\\
14 & 1& 32 & 0& 50 & 0& 68 & 1& 86 & 0& 104 & 0& 122 & 1\\
15 & 0& 33 & 0& 51 & 1& 69 & 0& 87 & 0& 105 & 1& 123 & 1\\
16 & 0& 34 & 1& 52 & 0& 70 & 1& 88 & 0& 106 & 0&  & \\
17 & 1& 35 & 0& 53 & 0& 71 & 0& 89 & 1& 107 & 0&  & \\
\hline
\end{longtable}

}

}

	\subsection{Data for Section \ref{sec:s}}
All $16$ polynomials are of the form  
$$p_0(x)+p_3(x)y^3+p_6(x)y^6+p_9(x)y^9+p_{12}(x)y^{12}.$$
For $(i,j)=(1,0)$ and all $T\in\{M^e, M^o, W^e, W^o\}$, 
$$p_6(x)=p_9(x)=p_{12}(x)=0.$$
The coefficients $p_j(x)$ and the $2$ initial terms to determine the solutions uniquely
are given below.

 For $\phi(M^e,0,0)$:
 \begin{align*}
	 p_0(x)&=\left(a^{12} + a^{11} + a^{4} + a^{3}\right) x^{12} + \left(a^{10} + a^{9} + a^{7} + a^{5} + a^{3} + a^{2}\right) x^{10}  \\ &\quad +\left(a^{8} + a^{7} + a^{2} + a\right) x^{8} + \left(a^{6} + a^{5} + a^{3} + a + 1\right) x^{6},\\
	 p_3(x)&=\left(a^{9} + a\right) x^{9} + \left(a^{8} + 1\right) x^{8} + \left(a^{6} + a^{5} + a^{2} + a\right) x^{7} + \left(a^{5} + a\right) x^{5}, \\ 
	 &\quad + \left(a^{4} + 1\right) x^{4} + \left(a^{2} + a\right) x^{3},\\
	 p_6(x)&=\left(a^{9} + a\right) x^{9} + \left(a^{7} + a^{6} + a^{5} + a^{4} + a^{3} + a^{2} + a + 1\right) x^{8} \\&\quad + \left(a^{6} + a^{5} + a^{2} + a\right) x^{7} + \left(a^{6} + a^{5} + a^{4} + a^{2} + a + 1\right) x^{6} + \left(a^{4} + 1\right) x^{5} \\&\quad + \left(a^{4} + a^{3} + a^{2} + a\right) x^{4}+ \left(a^{2} + a\right) x^{3} + \left(a^{2} + a + 1\right) x^{2} + \left(a + 1\right) x + 1,\\
	 p_9(x)&=\left(a^{8} + 1\right) x^{8} + \left(a^{7} + a^{6} + a^{5} + a^{4} + a^{3} + a^{2} + a + 1\right) x^{7}\\&\quad + \left(a^{6} + a^{4} + a^{2} + 1\right) x^{6} + \left(a^{5} + a^{4} + a + 1\right) x^{5} + \left(a^{4} + 1\right) x^{4}\\&\quad + \left(a^{3} + a^{2} + a + 1\right) x^{3} + \left(a^{2} + 1\right) x^{2}+ \left(a + 1\right) x,\\
	 p_{12}(x)&=\left(a^{8} + 1\right) x^{8} + 1,
 \end{align*}
 and the initial terms are $[1,a]$.

 For $\phi(M^e,0,1)$:
 \begin{align*}
	 p_0(x)&=(a^{24} + a^{22} + a^{8} + a^{6}) x^{24} + (a^{22} + a^{21} + a^{6} + a^{5}) x^{22} + (a^{20} + a^{19} + a^{18} +\\ &\quad
	 a^{17} + a^{15} + a^{14} + a^{13} + a^{11} + a^{10} + a^{9} + a^{7} + a^{6} + a^{5} + a^{4}) x^{20} + (a^{18}
	\\&\quad + a^{17} + a^{16} + a^{15} + a^{14} + a^{13} + a^{12} + a^{11} + a^{10} + a^{9} + a^{8} + a^{7} + a^{6} + a^{5}\\
	 &\quad + a^{4} + a^{3}) x^{18} + (a^{15} + a^{12} + a^{11} + a^{10} + a^{8} + a^{7} + a^{6} + a^{3}) x^{16}
	 + (a^{12} \\&\quad + a^{11} + a^{10} + a^{9} + a^{6} + a^{5} + a^{4} + a^{3}) x^{14} + (a^{9} + a^{8} + a^{6} + a^{4} + a^{3}) x^{12}, \\
	 p_3(x)&=(a^{14} + a^{12} + a^{10} + a^{8} + a^{6} + a^{4} + a^{2} + 1) x^{14} + (a^{13} + a^{12} + a^{9} + a^{8}\\&\quad + a^{5} + a^{4} + a + 1) x^{13} + (a^{10} + a^{8} + a^{2} + 1) x^{10} + (a^{9} + a^{8} + a + 1) x^{9} ,\\
	 p_6(x)&=(a^{12} + a^{8} + a^{4} + 1) x^{12} + (a^{10} + a^{8} + a^{2} + 1) x^{10} + (a^{8} + 1) x^{8}\\&\quad + (a^{6} + a^{4} + a^{2} + 1) x^{6} ,\\
	 p_9(x)&=(a^{10} + a^{8} + a^{2} + 1) x^{10} + (a^{9} + a^{8} + a + 1) x^{9} + (a^{8} + 1) x^{8} + (a^{7} + a^{6}\\&\quad + a^{5} + a^{4} + a^{3} + a^{2} + a + 1) x^{7} + (a^{6} + a^{4} + a^{2} + 1) x^{6} + (a^{5} + a^{4} + a\\&\quad + 1) x^{5} + (a^{4} + 1) x^{4} + (a^{3} + a^{2} + a + 1) x^{3} ,\\
	 p_{12}(x)&=(a^{8} + 1) x^{8} + 1,
 \end{align*}
 and the initial terms are $[0,a]$.

 For $\phi(M^e,1,0)$:
 \begin{align*}
	 p_0(x)&=1,\\
	 p_3(x)&=(a^{2} + 1) x^{2} + 1 ,
 \end{align*}
 and the initial terms are $[1,0]$.

 For $\phi(M^e,1,1)$:
 \begin{align*}
	 p_0(x)&=(a^{12} + a^{11} + a^{4} + a^{3}) x^{12} + (a^{9} + a^{7} + a^{5} + a^{3}) x^{10} + (a^{6} + a^{5} + a^{4} + a^{3})\\&\quad x^{8} + a^{3} x^{6} ,\\
	 p_3(x)&=(a^{9} + a) x^{9} + (a^{8} + 1) x^{8} + (a^{7} + a^{4} + a^{3} + 1) x^{7} + (a^{5} + a) x^{5} + (a^{4} + 1) x^{4} \\&\quad+ (a^{3} + 1) x^{3} ,\\
	 p_6(x)&=(a^{9} + a) x^{9} + (a^{7} + a^{6} + a^{5} + a^{4} + a^{3} + a^{2} + a + 1) x^{8} + (a^{7} + a^{4} + a^{3}\\&\quad + 1) x^{7} + (a^{5} + a) x^{6} + (a^{4} + 1) x^{5} + (a^{4} + a^{3} + a^{2} + a) x^{4} + (a^{3} + 1) x^{3}\\&\quad + a x^{2} + (a + 1) x + 1 ,\\
	 p_9(x)&=(a^{8} + 1) x^{8} + (a^{7} + a^{6} + a^{5} + a^{4} + a^{3} + a^{2} + a + 1) x^{7} + (a^{6} + a^{4} + a^{2}\\&\quad + 1) x^{6} + (a^{5} + a^{4} + a + 1) x^{5} + (a^{4} + 1) x^{4} + (a^{3} + a^{2} + a + 1) x^{3}\\&\quad + (a^{2} + 1) x^{2} + (a + 1) x ,\\
	 p_{12}(x)&=(a^{8} + 1) x^{8} + 1,
 \end{align*}
 and the initial terms are $[0,a]$.

 For $\phi(M^o,0,0)$:
 \begin{align*}
	 p_0(x)&=(a^{9} + a^{8} + a + 1) x^{12} + (a^{8} + a^{7} + a^{5} + a^{3} + a + 1) x^{10} + (a^{7} + a^{6} + a + 1)\\&\quad x^{8} + (a^{6} + a^{5} + a^{3} + a + 1) x^{6} ,\\
	 p_3(x)&=(a^{7} + a^{6} + a^{5} + a^{4} + a^{3} + a^{2} + a + 1) x^{8} + (a^{7} + a^{5} + a^{3} + a) x^{7} + (a^{6}\\&\quad + a^{2}) x^{6} + (a^{5} + a^{4} + a^{3} + a^{2}) x^{5} + (a^{4} + a^{3} + a + 1) x^{4}+ (a^{2} + a) x^{3} ,\\
	 p_6(x)&=(a^{6} + a^{4} + a^{2} + 1) x^{7} + (a^{6} + a^{5} + a^{2} + a) x^{6} + (a^{3} + a^{2} + a + 1) x^{5}\\&\quad + (a^{2} + 1) x^{4} + (a^{3} + 1) x^{3} + a x^{2} + (a + 1) x + 1 ,\\
	 p_9(x)&=(a^{5} + a^{4} + a + 1) x^{5} + (a + 1) x ,\\
	 p_{12}(x)&=(a^{4} + 1) x^{4} + 1 ,
 \end{align*}
 and the initial terms are $[1,a]$.

 For $\phi(M^o,0,1)$:
 \begin{align*}
	 p_0(x)&=(a^{21} + a^{19} + a^{5} + a^{3}) x^{24} + (a^{20} + a^{4}) x^{22} + (a^{19} + a^{16} + a^{12} + a^{8} + a^{4}\\&\quad
	 + a^{3}) x^{20} + (a^{18} + a^{14} + a^{10} + a^{6}) x^{18} + (a^{15} + a^{14} + a^{9} + a^{6} + a^{5} + a^{3})\\&\quad x^{16} + (a^{12} + a^{10} + a^{6} + a^{4}) x^{14} + (a^{9} + a^{8} + a^{6} + a^{4} + a^{3}) x^{12} ,\\
	 p_3(x)&=(a^{16} + 1) x^{16} + (a^{15} + a^{14} + a^{13} + a^{12} + a^{11} + a^{10} + a^{9} + a^{8} + a^{7} + a^{6} \\&\quad+ a^{5} + a^{4} + a^{3} + a^{2} + a + 1) x^{15} + (a^{14} + a^{12} + a^{10} + a^{8} + a^{6} + a^{4}\\&\quad + a^{2} + 1) x^{14} + (a^{13} + a^{12} + a^{9} + a^{8} + a^{5} + a^{4} + a + 1) x^{13} + (a^{12} + a^{8}\\&\quad + a^{4} + 1) x^{12} + (a^{11} + a^{10} + a^{9} + a^{8} + a^{3} + a^{2} + a + 1) x^{11} + (a^{10} + a^{8}\\&\quad + a^{2} + 1) x^{10} + (a^{9} + a^{8} + a + 1) x^{9} ,\\
	 p_6(x)&=(a^{12} + a^{8} + a^{4} + 1) x^{12} + (a^{10} + a^{8} + a^{2} + 1) x^{10} + (a^{8} + 1) x^{8} + (a^{6} + a^{4} +\\&\quad a^{2} + 1) x^{6} ,\\
	 p_9(x)&=(a^{8} + 1) x^{8} + (a^{7} + a^{6} + a^{5} + a^{4} + a^{3} + a^{2} + a + 1) x^{7} + (a^{4} + 1) x^{4} + (a^{3}\\&\quad + a^{2} + a + 1) x^{3} ,\\
	 p_{12}(x)&=(a^{4} + 1) x^{4} + 1 ,
 \end{align*}
 and the initial terms are $[0,a]$.

 For $\phi(M^o,1,0)$:
 \begin{align*}
	 p_0(x)&=1\\
	 p_3(x)&=(a + 1) x + 1 ,
 \end{align*}
 and the initial terms are $[1,a+1]$.

 For $\phi(M^o,1,1)$:
 \begin{align*}
	 p_0(x)&=(a^{12} + a^{11} + a^{4} + a^{3}) x^{12} + (a^{9} + a^{7} + a^{5} + a^{3}) x^{10} + (a^{6} + a^{5} + a^{4} + a^{3})\\&\quad x^{8} + a^{3} x^{6} ,\\
	 p_3(x)&=(a^{8} + a^{7} + a^{6} + a^{5} + a^{4} + a^{3} + a^{2} + a) x^{8} + (a^{6} + a^{4} + a^{2} + 1) x^{7} + (a^{6} +\\&\quad a^{2}) x^{6} + (a^{5} + a^{4} + a^{3} + a^{2}) x^{5} + (a^{3} + a) x^{4} + (a^{3} + 1) x^{3} ,\\
	 p_6(x)&=(a^{7} + a^{5} + a^{3} + a) x^{7} + (a^{5} + a^{4} + a + 1) x^{6} + (a^{3} + a^{2} + a + 1) x^{5}\\&\quad + (a^{2} + 1) x^{4} + (a^{2} + a) x^{3} + (a^{2} + a + 1) x^{2} + (a + 1) x + 1 ,\\
	 p_9(x)&=(a^{5} + a^{4} + a + 1) x^{5} + (a + 1) x ,\\
	 p_{12}(x)&=(a^{4} + 1) x^{4} + 1 ,
 \end{align*}
 and the initial terms are $[0,a]$.

 For $\phi(W^e,0,0)$:
 \begin{align*}
	 p_0(x)&=(a^{9} + a^{8} + a + 1) x^{12} + (a^{8} + a^{7} + a^{5} + a^{3} + a + 1) x^{10} + (a^{7} + a^{6} + a + 1)\\&\quad x^{8}+ (a^{6} + a^{5} + a^{3} + a + 1) x^{6} ,\\
	 p_3(x)&=(a^{8} + 1) x^{9} + (a^{8} + 1) x^{8} + (a^{6} + a^{5} + a^{2} + a) x^{7} + (a^{4} + 1) x^{5} + (a^{4} + 1) x^{4}\\&\quad + (a^{2} + a) x^{3} ,\\
	 p_6(x)&=(a^{8} + 1) x^{9} + (a^{8} + a^{7} + a^{6} + a^{5} + a^{4} + a^{3} + a^{2} + a) x^{8} + (a^{6} + a^{5} + a^{2}\\&\quad + a) x^{7} + (a^{6} + a^{5} + a^{4} + a^{2} + a + 1) x^{6} + (a^{5} + a) x^{5} + (a^{3} + a^{2} + a + 1)\\&\quad x^{4} + (a^{2} + a) x^{3} + (a^{2} + a + 1) x^{2} + (a + 1) x + 1 ,\\
	 p_9(x)&=(a^{8} + 1) x^{8} + (a^{7} + a^{6} + a^{5} + a^{4} + a^{3} + a^{2} + a + 1) x^{7} + (a^{6} + a^{4} + a^{2}\\&\quad + 1) x^{6} + (a^{5} + a^{4} + a + 1) x^{5} + (a^{4} + 1) x^{4} + (a^{3} + a^{2} + a + 1) x^{3}\\&\quad + (a^{2} + 1) x^{2} + (a + 1) x ,\\
	 p_{12}(x)&=(a^{8} + 1) x^{8} + 1 ,
 \end{align*}
 and the initial terms are $[1,1]$.

 For $\phi(W^e,0,1)$:
 \begin{align*}
	 p_0(x)&=(a^{18} + a^{16} + a^{2} + 1) x^{24} + (a^{17} + a^{16} + a + 1) x^{22} + (a^{16} + a^{15} + a^{14} + a^{13}\\&\quad + a^{11} + a^{10} + a^{9} + a^{7} + a^{6} + a^{5} + a^{3} + a^{2} + a + 1) x^{20} + (a^{15} + a^{14} + a^{13}\\&\quad + a^{12} + a^{11} + a^{10} + a^{9} + a^{8} + a^{7} + a^{6} + a^{5} + a^{4} + a^{3} + a^{2} + a + 1) x^{18}\\&\quad + (a^{13} + a^{10} + a^{9} + a^{8} + a^{6} + a^{5} + a^{4} + a) x^{16} + (a^{11} + a^{10} + a^{9} + a^{8}\\&\quad + a^{5} + a^{4} + a^{3} + a^{2}) x^{14} + (a^{9} + a^{8} + a^{6} + a^{4} + a^{3}) x^{12} ,\\
	 p_3(x)&=(a^{14} + a^{12} + a^{10} + a^{8} + a^{6} + a^{4} + a^{2} + 1) x^{14} + (a^{13} + a^{12} + a^{9} + a^{8} + a^{5}\\&\quad + a^{4} + a + 1) x^{13} + (a^{10} + a^{8} + a^{2} + 1) x^{10} + (a^{9} + a^{8} + a + 1) x^{9} ,\\
	 p_6(x)&=(a^{12} + a^{8} + a^{4} + 1) x^{12} + (a^{10} + a^{8} + a^{2} + 1) x^{10} + (a^{8} + 1) x^{8} + (a^{6} + a^{4}\\&\quad + a^{2} + 1) x^{6} ,\\
	 p_9(x)&=(a^{10} + a^{8} + a^{2} + 1) x^{10} + (a^{9} + a^{8} + a + 1) x^{9} + (a^{8} + 1) x^{8} + (a^{7} + a^{6}\\&\quad + a^{5} + a^{4} + a^{3} + a^{2} + a + 1) x^{7} + (a^{6} + a^{4} + a^{2} + 1) x^{6} + (a^{5} + a^{4} + a\\&\quad + 1) x^{5} + (a^{4} + 1) x^{4} + (a^{3} + a^{2} + a + 1) x^{3} ,\\
	 p_{12}(x)&=(a^{8} + 1) x^{8} + 1 ,
 \end{align*}
 and the initial terms are $[0,1]$.

 For $\phi(W^e,1,0)$:
 \begin{align*}
	 p_0(x)&=1\\
	 p_3(x)&=(a^{2} + 1) x^{2} + 1 ,\\
 \end{align*}
 and the initial terms are $[1,0]$.

 For $\phi(W^e,1,1)$:
 \begin{align*}
	 p_0(x)&=(a^{9} + a^{8} + a + 1) x^{12} + (a^{7} + a^{5} + a^{3} + a) x^{10} + (a^{5} + a^{4} + a^{3} + a^{2}) x^{8}\\&\quad + a^{3} x^{6} ,\\
	 p_3(x)&=(a^{8} + 1) x^{9} + (a^{8} + 1) x^{8} + (a^{7} + a^{4} + a^{3} + 1) x^{7} + (a^{4} + 1) x^{5} + (a^{4} + 1) x^{4}\\&\quad + (a^{3} + 1) x^{3} ,\\
	 p_6(x)&=(a^{8} + 1) x^{9} + (a^{8} + a^{7} + a^{6} + a^{5} + a^{4} + a^{3} + a^{2} + a) x^{8} + (a^{7} + a^{4} + a^{3}\\&\quad + 1) x^{7} + (a^{5} + a) x^{6} + (a^{5} + a) x^{5} + (a^{3} + a^{2} + a + 1) x^{4} + (a^{3} + 1) x^{3}\\&\quad + a x^{2} + (a + 1) x + 1 ,\\
	 p_9(x)&=(a^{8} + 1) x^{8} + (a^{7} + a^{6} + a^{5} + a^{4} + a^{3} + a^{2} + a + 1) x^{7} + (a^{6} + a^{4} + a^{2}\\&\quad + 1) x^{6} + (a^{5} + a^{4} + a + 1) x^{5} + (a^{4} + 1) x^{4} + (a^{3} + a^{2} + a + 1) x^{3}\\&\quad + (a^{2} + 1) x^{2} + (a + 1) x ,\\
	 p_{12}(x)&=(a^{8} + 1) x^{8} + 1 ,
 \end{align*}
 and the initial terms are $[0,1]$.

 For $\phi(W^o,0,0)$:
 \begin{align*}
	 p_0(x)&=(a^{12} + a^{11} + a^{4} + a^{3}) x^{12} + (a^{10} + a^{9} + a^{7} + a^{5} + a^{3} + a^{2}) x^{10} + (a^{8}\\&\quad + a^{7} + a^{2} + a) x^{8} + (a^{6} + a^{5} + a^{3} + a + 1) x^{6} ,\\
	 p_3(x)&=(a^{8} + a^{7} + a^{6} + a^{5} + a^{4} + a^{3} + a^{2} + a) x^{8} + (a^{6} + a^{4} + a^{2} + 1) x^{7}\\&\quad + (a^{4} + 1) x^{6} + (a^{3} + a^{2} + a + 1) x^{5} + (a^{4} + a^{3} + a + 1) x^{4} + (a^{2} + a) x^{3} ,\\
	 p_6(x)&=(a^{7} + a^{5} + a^{3} + a) x^{7} + (a^{5} + a^{4} + a + 1) x^{6} + (a^{5} + a^{4} + a^{3} + a^{2}) x^{5}\\&\quad + (a^{4} + a^{2}) x^{4} + (a^{3} + 1) x^{3} + a x^{2} + (a + 1) x + 1 ,\\
	 p_9(x)&=(a^{5} + a^{4} + a + 1) x^{5} + (a + 1) x ,\\
	 p_{12}(x)&=(a^{4} + 1) x^{4} + 1 ,
 \end{align*}
 and the initial terms are $[1,1]$.

 For $\phi(W^o,0,1)$:
 \begin{align*}
	 p_0(x)&=(a^{21} + a^{19} + a^{5} + a^{3}) x^{24} + (a^{18} + a^{2}) x^{22} + (a^{17} + a^{16} + a^{12} + a^{8} + a^{4}\\&\quad + a) x^{20} + (a^{12} + a^{8} + a^{4} + 1) x^{18} + (a^{13} + a^{11} + a^{10} + a^{7} + a^{2} + a) x^{16}\\&\quad + (a^{10} + a^{8} + a^{4} + a^{2}) x^{14} + (a^{9} + a^{8} + a^{6} + a^{4} + a^{3}) x^{12} ,\\
	 p_3(x)&=(a^{16} + 1) x^{16} + (a^{15} + a^{14} + a^{13} + a^{12} + a^{11} + a^{10} + a^{9} + a^{8} + a^{7} + a^{6}\\&\quad + a^{5} + a^{4} + a^{3} + a^{2} + a + 1) x^{15} + (a^{14} + a^{12} + a^{10} + a^{8} + a^{6} + a^{4} + a^{2}\\&\quad + 1) x^{14} + (a^{13} + a^{12} + a^{9} + a^{8} + a^{5} + a^{4} + a + 1) x^{13} + (a^{12} + a^{8} + a^{4} \\&\quad+ 1) x^{12} + (a^{11} + a^{10} + a^{9} + a^{8} + a^{3} + a^{2} + a + 1) x^{11} + (a^{10} + a^{8} + a^{2}\\&\quad + 1) x^{10} + (a^{9} + a^{8} + a + 1) x^{9} ,\\
	 p_6(x)&=(a^{12} + a^{8} + a^{4} + 1) x^{12} + (a^{10} + a^{8} + a^{2} + 1) x^{10} + (a^{8} + 1) x^{8} + (a^{6} + a^{4}\\&\quad + a^{2} + 1) x^{6} ,\\
	 p_9(x)&=(a^{8} + 1) x^{8} + (a^{7} + a^{6} + a^{5} + a^{4} + a^{3} + a^{2} + a + 1) x^{7} + (a^{4} + 1) x^{4} + (a^{3}\\&\quad + a^{2} + a + 1) x^{3} ,\\
	 p_{12}(x)&=(a^{4} + 1) x^{4} + 1 ,
 \end{align*}
 and the initial terms are $[0,1]$.

 For $\phi(W^o,1,0)$:
 \begin{align*}
	 p_0(x)&=1,\\
	 p_3(x)&=(a + 1) x + 1 ,
 \end{align*}
 and the initial terms are $[1,a+1]$.

 For $\phi(W^o,1,1)$:
 \begin{align*}
	 p_0(x)&=(a^{9} + a^{8} + a + 1) x^{12} + (a^{7} + a^{5} + a^{3} + a) x^{10} + (a^{5} + a^{4} + a^{3} + a^{2}) x^{8} \\&\quad+ a^{3} x^{6} ,\\
	 p_3(x)&=(a^{7} + a^{6} + a^{5} + a^{4} + a^{3} + a^{2} + a + 1) x^{8} + (a^{7} + a^{5} + a^{3} + a) x^{7} + (a^{4}\\&\quad + 1) x^{6} + (a^{3} + a^{2} + a + 1) x^{5} + (a^{3} + a) x^{4} + (a^{3} + 1) x^{3} ,\\
	 p_6(x)&=(a^{6} + a^{4} + a^{2} + 1) x^{7} + (a^{6} + a^{5} + a^{2} + a) x^{6} + (a^{5} + a^{4} + a^{3} + a^{2}) x^{5}\\&\quad + (a^{4} + a^{2}) x^{4} + (a^{2} + a) x^{3} + (a^{2} + a + 1) x^{2} + (a + 1) x + 1 ,\\
	 p_9(x)&=(a^{5} + a^{4} + a + 1) x^{5} + (a + 1) x ,\\
	 p_{12}(x)&=(a^{4} + 1) x^{4} + 1 ,
 \end{align*}
 and the initial terms are $[0,1]$.

 \subsection{Data for subsection \ref{subsection:pd1}}
All $16$ polynomials are of the form  
$$p_0(x)+p_1(x)y+p_2(x)y^2+p_{4}(x)y^{4}.$$
The coefficients $p_j(x)$ and the two initial terms to determine 
the solutions uniquely are given below.

For $\phi(A^e,0,0)$:
\begin{align*}
p_{0}(x)&=x^8 + x^6 + x^5 + x^2 + 1\\
p_{1}(x)&=x^2 + x,\\
p_{2}(x)&=x,\\
p_{4}(x)&=1.
\end{align*}
and the initial terms are [1, 0].

For $\phi(A^e,0,1)$:
\begin{align*}
	p_{0}(x)&=x^{13} + x^9 + x^4 + x^3 + x^2,\\
p_{1}(x)&=x^3 + x^2 + x + 1,\\
p_{2}(x)&=0,\\
p_{4}(x)&=x^5.
\end{align*}
and the initial terms are [0, 0].

For $\phi(A^e,1,0)$:
\begin{align*}
p_{0}(x)&=x^2,\\
p_{1}(x)&=1.
\end{align*}
and the initial terms are [0, 0].

For $\phi(A^e,1,1)$:
\begin{align*}
p_{0}(x)&=x^8 + x^6 + x^5,\\
p_{1}(x)&=x^2 + x,\\
p_{2}(x)&=x,\\
p_{4}(x)&=1.
\end{align*}
and the initial terms are [0, 0].

For $\phi(A^o,0,0)$:
\begin{align*}
p_{0}(x)&=x^6 + x^4 + x^3 + x^2 + 1,\\
p_{1}(x)&=x+1,\\
p_{2}(x)&=x,\\
p_{4}(x)&=x^2.
\end{align*}
and the initial terms are [1, 0].

For $\phi(A^o,0,1)$:
\begin{align*}
	p_{0}(x)&=x^{11} + x^5 + x^4 + x^3 + x^2,\\
p_{1}(x)&=x^3 + x^2 + x + 1,\\
p_{2}(x)&=0,\\
	p_{4}(x)&=x^7.
\end{align*}
and the initial terms are [0, 0].

For $\phi(A^o,1,0)$:
\begin{align*}
p_{0}(x)&=x,\\
p_{1}(x)&=1.
\end{align*}
and the initial terms are [0, 1].
For $\phi(A^o,1,1)$:
\begin{align*}
p_{0}(x)&=x^6 + x^4 + x^3,\\
p_{1}(x)&=x+1,\\
p_{2}(x)&=x,\\
p_{4}(x)&=x^2.
\end{align*}
and the initial terms are [0, 0].

For all $0\leq i,j\leq 1$, $\phi(B^e,i,j)=\phi(A^o,i,j)$, $\phi(B^o,i,j)=
\phi(Ae,i,j)$. The corresponding initial conditions are also the same.
We happen to have $Ae=Bo$ and $Ao=Be$ here.

\subsection{Data for subsection \ref{subsection:pd2}}
All $16$ polynomials are of the form  
$$p_0(x)+p_3(x)y^3+p_6(x)y^6+p_9(x)y^9+p_{12}(x)y^{12}.$$
For $(i,j)=(1,0)$ and all $T\in\{A^e, A^o, B^e, B^o\}$, 
$$p_6(x)=p_9(x)=p_{12}(x)=0.$$
The coefficients $p_j(x)$ and the initial terms to determine the solutions uniquely
are given below.

 For $\phi(A^e,0,0)$:
 \begin{align*}
	 p_0(x)&=x^{24} + x^{22} + x^{20} + x^{17} + x^{16} + x^{14} + x^{13} + x^{5} + x^{3} + x + 1,\\
	 p_3(x)&=x^{15} + x^{14} + x^{12} + x^{9} + x^{6} + x^{2},\\
	 p_6(x)&= x^{11} + x^{9} + x^{8} + x^{7} + x^{5} + x^{3} + x^{2} + x,\\
	 p_9(x)&=  x^6 + x^4 + x^3 + x,\\
	 p_{12}(x)&=x^2 + x + 1.
 \end{align*}
 and the initial terms are $[1,1]$.

 For $\phi(A^e,0,1)$:
 \begin{align*}
	 p_0(x)&=x^{42} + x^{41} + x^{39} + x^{35} + x^{34} + x^{33} + x^{32} + x^{30} + x^{26} + x^{25} + x^{24}\\&\quad + x^{20} + x^{18} + x^{17} + x^{16} + x^{15} + x^{14} + x^{12} + x^{11} + x^{10} + x^{9},\\
	 p_3(x)&=x^{29} + x^{27} + x^{26} + x^{24} + x^{21} + x^{19} + x^{18} + x^{16} + x^{13} + x^{11} + x^{10}\\&\quad + x^{8} + x^{5} + x^{3} + x^{2} + 1,\\
	 p_6(x)&=x^{28} + x^{27} + x^{25} + x^{24} + x^{12} + x^{11} + x^{9} + x^{8},\\
	 p_9(x)&= x^{27} + x^{24} + x^{19} + x^{16}\\
	 p_{12}(x)&=x^{26} + x^{25} + x^{24}.
 \end{align*}
 and the initial terms are $[0,0,0 ,1]$.

 For $\phi(A^e,1,0)$:
 \begin{align*}
	 p_0(x)&=x^9,\\
	 p_3(x)&=x^2+x+1.
 \end{align*}
 and the initial terms are $[0,0,0 ,1]$.

 For $\phi(A^e,1,1)$:
 \begin{align*}
	 p_0(x)&=x^{30} + x^{29} + x^{28} + x^{25} + x^{24} + x^{23} + x^{21},\\
	 p_3(x)&=x^{21} + x^{19} + x^{17} + x^{16} + x^{15} + x^{14} + x^{11} + x^{9} + x^{8} + x^{7} + x^{4} + x^{3},\\
	 p_6(x)&=x^{17} + x^{16} + x^{15} + x^{14} + x^{13} + x^{10} + x^{7} + x^{5} + x^{3} + x^{2} ,\\
	 p_9(x)&=x^{12} + x^{11} + x^{10} + x^{9} + x^{8} + x^{7} + x^{6} + x^{5} + x^{4} + x^{3} + x^{2} + x,\\
	 p_{12}(x)&=x^{8} + x^{4} + 1.
 \end{align*}
 and the initial terms are $[0, 0, 0, 0, 0, 0, 1, 1]$.

 For $\phi(A^o,0,0)$:
 \begin{align*}
	 p_0(x)&=x^{30} + x^{28} + x^{27} + x^{25} + x^{24} + x^{23} + x^{22} + x^{20} + x^{14} + x^{13} + x^{7}\\ &\quad + x^{6} + x^{5} + x^{4} + 1,\\
	 p_3(x)&=x^{22} + x^{21} + x^{17} + x^{16} + x^{15} + x^{11} + x^{8} + x^{7} + x^{6} + x^{2} + x + 1,\\
	 p_6(x)&=x^{23} + x^{22} + x^{20} + x^{16} + x^{15} + x^{13} + x^{12} + x^{9} + x^{7} + x^{6} + x^{5} + x^{2},\\
	 p_9(x)&=x^{21} + x^{18} + x^{17} + x^{15} + x^{14} + x^{13} + x^{12} + x^{11} + x^{10} + x^{8} + x^{7} + x^{4},\\
	 p_{12}(x)&=x^{22} + x^{14} + x^{6}.
 \end{align*}
 and the initial terms are $[1,1]$.

 For $\phi(A^o,0,1)$:
 \begin{align*}
	 p_0(x)&=x^{48} + x^{47} + x^{43} + x^{42} + x^{41} + x^{40} + x^{39} + x^{36} + x^{35} + x^{34} + x^{33}\\ &\quad + x^{32} + x^{27} + x^{25} + x^{24} + x^{14} + x^{12} + x^{11} + x^{9},\\
	 p_3(x)&=x^{43} + x^{42} + x^{41} + x^{40} + x^{27} + x^{26} + x^{25} + x^{24} + x^{19} + x^{18} + x^{17}\\ &\quad + x^{16} + x^{3} + x^{2} + x + 1,\\
	 p_6(x)&=x^{44} + x^{42} + x^{36} + x^{34} + x^{20} + x^{18} + x^{12} + x^{10},\\
	 p_9(x)&=x^{45} + x^{44} + x^{21} + x^{20},\\
	 p_{12}(x)&=x^{46} + x^{38} + x^{30}.
 \end{align*}
 and the initial terms are $[0,0,0 ,1]$.
 
 For $\phi(A^o,1,0)$:
 \begin{align*}
	 p_0(x)&=x^6,\\
	 p_3(x)&=x^4+x^2+1.
 \end{align*}
 and the initial terms are $[0,0,1]$.

 For $\phi(A^o,1,1)$:
 \begin{align*}
	 p_0(x)&=x^{24} + x^{23} + x^{22} + x^{19} + x^{18} + x^{17} + x^{15},\\
	 p_3(x)&=x^{20} + x^{19} + x^{17} + x^{14} + x^{13} + x^{11} + x^{10} + x^{9} + x^{6} + x^{4} + x^{3} + 1,\\
	 p_6(x)&=x^{21} + x^{20} + x^{17} + x^{13} + x^{12} + x^{11} + x^{10} + x^{7} + x^{6} + x^{4} + x^{3} + x^{2},\\
	 p_9(x)&=x^{21} + x^{18} + x^{17} + x^{15} + x^{14} + x^{13} + x^{12} + x^{11} + x^{10} + x^{8} + x^{7} + x^{4},\\
	 p_{12}(x)&=x^{22} + x^{14} + x^{6}.
 \end{align*}
 and the initial terms are $[0,0,0,0,0,1 ]$.

 For $\phi(B^e,0,0)$:
 \begin{align*}
	 p_0(x)&=x^{30} + x^{28} + x^{27} + x^{25} + x^{24} + x^{23} + x^{22} + x^{20} + x^{14} + x^{13} + x^{7}\\ &\quad + x^{6} + x^{5} + x^{4} + 1,\\
	 p_3(x)&=x^{20} + x^{18} + x^{17} + x^{14} + x^{12} + x^{11} + x^{6} + 1,\\
	 p_6(x)&=x^{19} + x^{18} + x^{17} + x^{14} + x^{13} + x^{10} + x^{6} + x^{5} + x^{4} + x^{2},\\
	 p_9(x)&=x^{15} + x^{14} + x^{13} + x^{12} + x^{11} + x^{10} + x^{9} + x^{8} + x^{7} + x^{6} + x^{5} + x^{4},\\
	 p_{12}(x)&=x^{14} + x^{10} + x^{6}.
 \end{align*}
 and the initial terms are $[1,0]$.

 For $\phi(B^e,0,1)$:
 \begin{align*}
	 p_0(x)&=x^{48} + x^{47} + x^{43} + x^{42} + x^{41} + x^{40} + x^{39} + x^{36} + x^{35} + x^{34} + x^{33}\\ &\quad + x^{32} + x^{27} + x^{25} + x^{24} + x^{14} + x^{12} + x^{11} + x^{9},\\
	 p_3(x)&=x^{41} + x^{38} + x^{37} + x^{35} + x^{34} + x^{32} + x^{31} + x^{29} + x^{28} + x^{26} + x^{23}\\ &\quad + x^{21} + x^{20} + x^{18} + x^{15} + x^{13} + x^{12} + x^{10} + x^{9} + x^{7} + x^{6} + x^{4} + x^{3} + 1,\\
	 p_6(x)&=x^{40} + x^{36} + x^{34} + x^{32} + x^{30} + x^{26} + x^{24} + x^{20} + x^{18} + x^{16} + x^{14} + x^{10},\\
	 p_9(x)&=x^{39} + x^{36} + x^{35} + x^{32} + x^{27} + x^{24} + x^{23} + x^{20},\\
	 p_{12}(x)&=x^{38} + x^{34} + x^{30}.
 \end{align*}
 and the initial terms are $[0,0,0,1 ]$.

 For $\phi(B^e,1,0)$:
 \begin{align*}
	 p_0(x)&=x^6,\\
	 p_3(x)&=x^2+x+1.
 \end{align*}
 and the initial terms are $[0,0,1,1 ]$.

 For $\phi(B^e,1,1)$:
 \begin{align*}
	 p_0(x)&=x^{24} + x^{23} + x^{22} + x^{19} + x^{18} + x^{17} + x^{15},\\
	 p_3(x)&=x^{18} + x^{16} + x^{14} + x^{13} + x^{12} + x^{11} + x^{8} + x^{6} + x^{5} + x^{4} + x + 1,\\
	 p_6(x)&=x^{17} + x^{16} + x^{15} + x^{14} + x^{13} + x^{10} + x^{7} + x^{5} + x^{3} + x^{2},\\
	 p_9(x)&=x^{15} + x^{14} + x^{13} + x^{12} + x^{11} + x^{10} + x^{9} + x^{8} + x^{7} + x^{6} + x^{5} + x^{4},\\
	 p_{12}(x)&=x^{14} + x^{10} + x^{6}.
 \end{align*}
 and the initial terms are $[ 0,0,0,0,0,1 ]$.

 For $\phi(B^o,0,0)$:
 \begin{align*}
	 p_0(x)&=x^{24} + x^{22} + x^{20} + x^{17} + x^{16} + x^{14} + x^{13} + x^{5} + x^{3} + x + 1 ,\\
	 p_3(x)&=x^{17} + x^{13} + x^{12} + x^{11} + x^{10} + x^{9} + x^{8} + x^{7} + x^{6} + x^{4} + x^{3} + x^{2},\\
	 p_6(x)&=x^{15} + x^{12} + x^{11} + x^{10} + x^{9} + x^{8} + x^{7} + x^{6} + x^{5} + x^{4} + x^{2} + x,\\
	 p_9(x)&=x^{12} + x^{11} + x^{10} + x^{9} + x^{8} + x^{7} + x^{6} + x^{5} + x^{4} + x^{3} + x^{2} + x,\\
	 p_{12}(x)&=x^{10} + x^{9} + x^{8} + x^{6} + x^{5} + x^{4} + x^{2} + x + 1.
 \end{align*}
 and the initial terms are $[1,0]$.

 For $\phi(B^o,0,1)$:
 \begin{align*}
	 p_0(x)&=x^{42} + x^{41} + x^{39} + x^{35} + x^{34} + x^{33} + x^{32} + x^{30} + x^{26} + x^{25} + x^{24}\\ &\quad + x^{20} + x^{18} + x^{17} + x^{16} + x^{15} + x^{14} + x^{12} + x^{11} + x^{10} + x^{9},\\
	 p_3(x)&=x^{31} + x^{30} + x^{25} + x^{24} + x^{23} + x^{22} + x^{17} + x^{16} + x^{15} + x^{14} + x^{9}\\ &\quad + x^{8} + x^{7} + x^{6} + x + 1,\\
	 p_6(x)&=x^{32} + x^{31} + x^{30} + x^{26} + x^{25} + x^{24} + x^{16} + x^{15} + x^{14} + x^{10} + x^{9} + x^{8},\\
	 p_9(x)&=x^{33} + x^{32} + x^{29} + x^{28} + x^{21} + x^{20} + x^{17} + x^{16},\\
	 p_{12}(x)&=x^{34} + x^{33} + x^{32} + x^{30} + x^{29} + x^{28} + x^{26} + x^{25} + x^{24}.
 \end{align*}
 and the initial terms are $[0,0,0,1 ]$.

 For $\phi(B^o,1,0)$:
 \begin{align*}
	 p_0(x)&=x^9,\\
	 p_3(x)&=x^4+x^2+1.
 \end{align*}
 and the initial terms are $[0,0,0,1 ]$.
 
 For $\phi(B^o,1,1)$:
 \begin{align*}
	 p_0(x)&=x^{30} + x^{29} + x^{28} + x^{25} + x^{24} + x^{23} + x^{21},\\
	 p_3(x)&=x^{23} + x^{22} + x^{20} + x^{17} + x^{16} + x^{14} + x^{13} + x^{12} + x^{9} + x^{7} + x^{6} + x^{3}
,\\
	 p_6(x)&=x^{21} + x^{20} + x^{17} + x^{13} + x^{12} + x^{11} + x^{10} + x^{7} + x^{6} + x^{4} + x^{3} + x^{2},\\
	 p_9(x)&=x^{18} + x^{15} + x^{14} + x^{12} + x^{11} + x^{10} + x^{9} + x^{8} + x^{7} + x^{5} + x^{4} + x,\\
	 p_{12}(x)&=x^{16} + x^{8} + 1.
 \end{align*}
 and the initial terms are $[0,0,0,0,0,0,1 ]$.

\bibliographystyle{plain}

\bibliography{article}

\end{document}